\RequirePackage{fix-cm}
\documentclass[smallextended]{svjour3}       
\smartqed  
\usepackage{graphicx}
\usepackage{amsmath}
\usepackage{amssymb}

\def \F {{\mathbb F}}

\def \N {{\mathbb N}}

\def \Tr {{\rm Tr_n}}
\def \T {{\rm Tr}}

\def \C {\mathcal{C}}
\def \D {\mathcal{D}}

\begin{document}

\title{Idempotent and $p$-potent quadratic functions: Distribution of nonlinearity and co-dimension
}


\author{ Nurdag\"ul Anbar  \and
        Wilfried Meidl \and Alev Topuzo\u glu
}


\institute{Nurdag\"ul Anbar \at
  Technical University of Denmark, Matematiktorvet, Building 303B, DK-2800, Lyngby, Denmark \\
              \email{nurdagulanbar2@gmail.com}           
           \and
           Wilfried Meidl \at
          Johann Radon Institute for Computational and Applied Mathematics, Austrian Academy of Sciences, Linz, Austria \\
          \email{meidlwilfried}
          \and
          Alev Topuzo\u glu \at
          Sabanc\i\ University, MDBF, Orhanl\i, Tuzla, 34956 \.Istanbul, Turkey. \\ 
          \email{alev@sabanciuniv.edu}
          }

\date{Received: date / Accepted: date}

\maketitle

\begin{abstract}
The Walsh transform $\widehat{Q}$ of a quadratic function
$Q:\F_{p^n}\rightarrow\F_p$ satisfies $|\widehat{Q}(b)| \in \{0,p^{\frac{n+s}{2}}\}$ 
for all $b\in\F_{p^n}$, where $0\le s\le n-1$ is an integer depending on $Q$.
In this article, we study the following three classes of quadratic functions of wide interest. The class
$\mathcal{C}_1$ is defined for arbitrary $n$ as
$\mathcal{C}_1 = \{Q(x) = \Tr(\sum_{i=1}^{\lfloor (n-1)/2\rfloor}a_ix^{2^i+1})\;:\; a_i \in\F_2\}$, and
the larger class $\mathcal{C}_2$ is defined for even $n$ as
$\mathcal{C}_2 = \{Q(x) = \Tr(\sum_{i=1}^{(n/2)-1}a_ix^{2^i+1}) + {\rm Tr_{n/2}}(a_{n/2}x^{2^{n/2}+1})
\;:\; a_i \in\F_2\}$. For an odd prime $p$, the subclass $\mathcal{D}$ of all $p$-ary quadratic functions is defined as
$\mathcal{D} = \{Q(x) = \Tr(\sum_{i=0}^{\lfloor n/2\rfloor}a_ix^{p^i+1})\;:\; a_i \in\F_p\}$. 
We determine the distribution of the parameter $s$ for $\mathcal{C}_1, \mathcal{C}_2$ and $\mathcal{D}$. 
As a consequence we obtain the distribution of the nonlinearity for the rotation symmetric quadratic 
Boolean functions, and in the case $p > 2$, our results yield the distribution of the co-dimensions for the 
rotation symmetric quadratic $p$-ary functions, which have been attracting considerable attention recently.
We also present the complete weight distribution of the subcodes of the second order Reed-Muller codes
corresponding to $\mathcal{C}_1$ and $\mathcal{C}_2$.
\keywords{Quadratic functions \and plateaued functions \and bent functions \and Walsh transform \and idempotent functions 
\and rotation symmetric \and Reed-Muller code}
\subclass{11T06 \and 11T71 \and 11Z05}
\end{abstract}


\section{Introduction}

Omitting the linear and constant terms, a quadratic function $Q$ from $\F_{p^n}$ to $\F_p$, for a prime $p$,  
can be expressed in trace form as
\begin{equation}
\label{quadra} 
Q(x) = \Tr\left(\sum_{i=0}^{\lfloor{n}/2\rfloor}a_ix^{p^i+1}\right)\ , \quad a_i\in\F_{p^n} \ , 
\end{equation}
where $\T_n$ denotes the absolute trace from $\F_{p^n}$ to $\F_p$.
When $n$ is odd, this representation is unique. For even $n$ the coefficient $a_{n/2}$ is taken modulo 
the additive subgroup $G = \{x\in\F_{p^n}\;:\;\T^n_{n/2}(x) = 0\}$ of $\F_{p^n}$, where $\T^n_{n/2}$
denotes the relative trace from $\F_{p^n}$ to $\F_{p^{n/2}}$. Furthermore, if $p=2$, then $a_0 = 0$,
hence $i$ in  the summation in $(\ref{quadra})$ ranges over 
$1 \leq i \leq \lfloor{n}/2\rfloor$, since $a_0x^2$ is a linear term.

The {\it Walsh transform} $\widehat{f}$ of a function $f:\F_{p^n}\rightarrow \F_p$ is the function from $\F_{p^n}$ into
the set of complex numbers defined as
\[ \widehat{f}(b) = \sum_{x \in \mathbb F_{p^n}}\epsilon_p^{f(x)-\Tr(bx)} \ , \]
where $\epsilon_p = e^{2\pi i/p}$ is a complex $p$-th root of unity.

Quadratic functions belong to the class of {\it plateaued functions}, for which for every $b \in \F_{p^n}$, the
Walsh transform $\widehat{f}(b)$ vanishes or has absolute value $p^{(n+s)/2}$ for some fixed integer $0\le s \le n$.
Accordingly we call $f$ {\it $s$-plateaued}. 
Note that if $p=2$, then $\epsilon_p = -1$, and $\widehat{f}(b)$ is an integer. Hence for any $s$-plateaued function
from $\F_{2^n}$ to $\F_2$, $n$ and $s$ must be of the same parity. Recall that a $0$-plateaued function from
$\F_{p^n}$ to $\F_p$ is called {\it bent}. Clearly a Boolean bent function can exist only when $n$ is even. When $p$ is odd, a $1$-plateaued function is called {\it semi-bent}.
A Boolean function $f$ is called semi-bent, if $f$ is $1$ or $2$-plateaued, depending on the parity of $n$. 

The {\it nonlinearity} $N_f$ of a function $f:\F_{p^n}\rightarrow\F_p$ is defined to be the smallest Hamming distance of $f$ 
to any affine function, i.e.
\[ N_f = \min_{u\in\F_{p^n},v\in\F_p}|\{x\in\F_{p^n}\;:\;f(x) \ne \Tr(ux)+v\}| \ . \]
The nonlinearity of a Boolean function $f$ can be expressed in terms of the Walsh transform as
\begin{equation}
\label{Nf}
N_f = 2^{n-1} - \frac{1}{2}\max_{b\in\F_{2^n}}|\widehat{f}(b)| \ .
\end{equation}
By Parseval's identity we have $\sum_{b\in\F_{2^n}}\left|\widehat{f}(b)\right|^2=2^{2n}$ for any Boolean function $f$. 
As a consequence, Boolean bent functions are the Boolean functions attaining the highest possible nonlinearity.
Since high nonlinearity is crucial for cryptographic applications, Boolean bent functions are of particular interest.

Recall that the $r$th order Reed-Muller code $R(r,n)$ of length $2^n$ is defined as
\[ R(r,n)=\{(f(\alpha_{1}),f(\alpha_{2}),\cdots,f(\alpha_{2^{n}}))\;
|\; f\in P_{r}\} \ , \]
where $P_{r}$ is the set of all polynomials from $\F_{2^n}$ to $\F_2$ (or from $\F_2^n$ to $\F_2$)
of algebraic degree at most $r$, and $\alpha_{1},\alpha_{2},\ldots,\alpha_{2^n}$ are the elements of
$\F_{2^n}$ (or $\F_2^n$) in some fixed order.
The set of quadratic Boolean functions together with the constant and affine functions form the second order
Reed-Muller codes.

In this work we focus on the subclasses of the set of quadratic functions given in Equation $(\ref{quadra})$, obtained by
restricting the coefficients to the prime subfield. Namely, for $p=2$ we consider $ \C_1$ and the larger class 
$\C_2$ defined as
\[ \mathcal{C}_1 = \left\{Q(x) = \Tr\left(\sum_{i=1}^{\lfloor (n-1)/2\rfloor}a_ix^{2^i+1}\right)\;:\; a_i \in\F_2 ,\; 1\le i\le \left\lfloor \dfrac{n-1}{2}\right\rfloor \right\} , \; \text{and} \]
\begin{align*}
& \mathcal{C}_2 = \left\{Q(x) = \Tr\left(\sum_{i=1}^{(n/2)-1}a_ix^{2^i+1}\right) + {\rm Tr_{n/2}}(a_{n/2}x^{2^{n/2}+1})\right. : \\
& \qquad \qquad \qquad  \qquad \qquad \qquad \qquad \qquad \qquad \qquad 
a_{n/2},a_i \in\F_2,\; 1\le i\le n/2 \Bigg\} \ .
\end{align*}
Note that $ \C_1$ is defined for arbitrary $n$, while $\C_2$ is defined for even $n$ only. 
These two classes of Boolean functions have attracted significant attention in the last decade, 
 see the articles \cite{caw,cpt,f2,hf,kgs,kkos,mt,mrt,yg}.
For $p>2$ we put
\[ \mathcal{D} =\left\{Q(x) = \Tr\left(\sum_{i=0}^{\lfloor n/2\rfloor}a_ix^{p^i+1}\right)\;:\; a_i \in\F_p\right\}\ . \]
The class $\D$ has also been studied previously, see \cite{caw,kgs,lhz,mt,mrt}. 

%

The study of the Walsh spectrum of quadratic functions in $\C_1$ and $\mathcal{D}$
has been initiated in \cite{kgs}, where the authors determine all $n$ for which all such quadratic functions are semi-bent. 
This result was extended in \cite{cpt} for $\C_1$. In the articles \cite{hf,yg} bent functions in $\C_2$ are constructed. 

Enumeration of functions in $\C_1$, $\C_2$ and $\D$ for particular values of $s$ has been a challenging problem. 
In \cite{yg} the number of bent functions in $\C_2$ was obtained for some special classes of $n$. Counting
results on Boolean quadratic functions in $\C_1$ with a large value of $s$ 
have been obtained in the paper \cite{f2}. Far reaching enumeration results for the sets $\C_1$ and $\mathcal{D}$
have been obtained in \cite{caw,mt,mrt} by the use of 
methods originally employed in the analysis of the linear
complexity of periodic sequences, see \cite{fno}. 

Let us denote the number of $s$-plateaued quadratic functions in $\C_1$ and $\mathcal{D}$ by $\mathcal{N}_n(s)$ 
and $\mathcal{N}_n^{(p)}(s)$, respectively.
In \cite{mt}, the number $\mathcal{N}_n(s)$
has been determined for $n = 2^m$, $m\ge 1$, and all possible values of $s$.
In \cite{mrt}, $\mathcal{N}_n(s)$ and $\mathcal{N}_n^{(p)}(s)$ are described by the use of
{\it generating polynomials} $\mathcal{G}_n(z)$ and $\mathcal{G}_n^{(p)}(z)$, defined by
\[ \mathcal{G}_n(z) = \sum_{t=0}^n\mathcal{N}_n(n-t)z^t\quad\mbox{and}\quad\mathcal{G}_n^{(p)}(z) = \sum_{t=0}^n\mathcal{N}_n^{(p)}(n-t)z^t \ .\]
In the cases of odd $n$ and $n = 2m$ for an odd $m$, the polynomial $\mathcal{G}_n(z)$ is
determined as a product of polynomials, see \cite{mrt}. 
Similarly $\mathcal{G}_n^{(p)}(z)$ is obtained for $n$ with $\gcd(n,p) = 1$ 
in \cite{cm,mrt}. In particular, explicit formulas for the number of bent functions in $\mathcal{D}$ and of semi-bent functions 
in $\C_1$ are given for such $n$. For a result on the number of semi-bent functions in $\C_2$ we may refer to the recent article \cite{kkos}.
The average behaviour of the Walsh transform of functions in $\C_1$ and $\mathcal{D}$ is analysed in \cite{caw}.
We remark that unlike $\C_2$, the set $\C_1$ does not contain bent functions.

In this work we present the solution of the enumeration problem in all remaining cases, i.e., we extend the above results to functions in $\C_1$, $\C_2$, $\mathcal{D}$ {\it for arbitrary $n$ and all possible $s$}, by determining,
\begin{itemize}
\item[(i)] the generating polynomial $\mathcal{G}_n(z)$ for any (even) number $n$,
\item[(ii)] the generating polynomial $\mathcal{H}_n(z) = \sum_{t=0}^n\mathcal{M}_n(n-t)z^t$ concerning the number $\mathcal{M}_n(s)$ 
of $s$-plateaued functions in $\C_2$, for any even number $n$, and
\item[(ii)] the generating polynomial $\mathcal{G}_n^{(p)}(z)$ for any number $n$, $\gcd(n,p) > 1$.
\end{itemize}
We therefore completely describe the distribution of the parameter $s$ in the set $\C_2$, and in the sets $\C_1$ and $\mathcal{D}$ 
in all remaining cases. This also yields the distribution of the nonlinearity in $\C_1$ and $\C_2$ in full generality.
In particular, one can obtain the number of bent functions in the sets $\C_2$ and $\mathcal{D}$,
or the number of semi-bent functions in $\C_1$, $\C_2$ and $\mathcal{D}$ for arbitrary integers $n$.

\begin{remark} 
\label{rot.symm} The classes $\mathcal{C}_1, \mathcal{C}_2$ and $\mathcal{D}$ have additional properties that we explain below. 
A Boolean function $f$, defined over $\F_{2^n}$ is {\it idempotent} if it satisfies $f(x^2) = f(x)$ for all $x\in\F_{2^n}$.
The set of idempotent quadratic functions coincides with $\mathcal{C}_1$ when $n$ is odd, and it coincides with $\mathcal{C}_2$
when $n$ is even.  For an odd prime $p$, one can similarly define a {\it $p$-potent} function as   
a function $f$ from $\F_{p^n}$ to $\F_p$ that satisfies $f(x^p) = f(x)$ for all $x\in\F_{p^n}$. As one would expect,
the set of $p$-potent quadratic functions coincides with $\mathcal{D}$.
It is observed in \cite{cgl} that there is a nonlinearity preserving one-to-one correspondence between the set of idempotent quadratic functions 
from $\F_{2^n}$ to $\F_2$ and the set of {\it rotation symmetric} quadratic functions from $\F_2^n$ to $\F_2$.
(Note that only even $n$ is considered in \cite{cgl}, but the same applies to the case of odd $n$.) 
Following the arguments of \cite{cgl}, one can show that this property extends to any prime $p \geq 2$ and hence
there is a one-to-one correspondence between the set
$\mathcal{D}$ and the set of rotation symmetric quadratic functions from $\F_p^n$ to $\F_p$, which preserves the parameter $s$.
Hence many results on
idempotent and $p$-potent quadratic functions also yield results on rotation symmetric quadratic functions, which add to the interest in the classes $\mathcal{C}_1$,
$\mathcal{C}_2$ and $\mathcal{D}$.
\end{remark}
Having obtained the distribution of $s$ in $\mathcal{C}_1$,
$\mathcal{C}_2$, $\mathcal{D}$, we are able to give the distribution of the nonlinearity of rotation symmetric quadratic functions from $\F_2^n$ to $\F_2$, and the
distribution of the co-dimension of rotation symmetric quadratic functions from $\F_p^n$ to $\F_p$, for odd $p$.
We also analyse the subcodes of the second order Reed-Muller code obtained from $\C_1$ and $\C_2$, and 
present the weight distribution for both subcodes of $R(2,n)$.

\section{Preliminaries}
\label{prel}

In this section we summarize basic tools that we use to obtain our results.
We essentially follow the notation of \cite{mt,mrt}. 
For technical reasons we include the $0$-function, for which all coefficients $a_i$ are zero, in all sets
$\mathcal{C}_1$, $\mathcal{C}_2$ and $\mathcal{D}$. Being constant, the zero function is $n$-plateaued.

Let $Q(x)$ be in $\C_1$, i.e. $ Q(x)= \Tr(\sum_{i=1}^{\lfloor(n-1)/2\rfloor}a_ix^{2^i+1})$, $a_i \in\F_2$. 
We associate to $Q$, the polynomial
\[ A(x) = \sum_{i=1}^{\lfloor(n-1)/2\rfloor}(a_ix^i+a_ix^{n-i}) \]
of degree at most $n-1$. For even $n$ we consider $Q(x) \in \C_2$, i.e. $Q(x) = \Tr(\sum_{i=1}^{(n/2)-1}a_ix^{2^i+1}) + {\rm Tr_{n/2}}(a_{n/2}x^{2^{n/2}+1})$,
$a_i \in\F_2$, and the associated polynomial
\[ A(x) = \sum_{i=1}^{(n/2)-1}(a_ix^i+a_ix^{n-i}) + a_{n/2}x^{n/2} \]
of degree at most $n-1$. \\
When $p$ is odd and $n$ is arbitrary we consider $Q(x) \in \mathcal{D}$, i.e. the quadratic function of the form $ Q(x)= \Tr(\sum_{i=0}^{\lfloor n/2\rfloor}a_ix^{p^i+1})$, $a_i \in\F_p$, and the associated polynomial 
\[ A(x) = \sum_{i=0}^{\lfloor n/2\rfloor}(a_ix^i+a_ix^{n-i}) \]
of degree at most $n$.
With the standard Welch squaring method of the Walsh transform, one can easily see that in all three cases, the quadratic function $Q$ 
is $s$-plateaued if and only if
\[ s = \deg\big(\gcd(x^n-1,A(x))\big) \ , \]
see also \cite{hk,kgs,lhz,yg}.
We observe that $A(x) = x^dh(x)$, where $d$ is a non-negative integer and $h$ is a self-reciprocal polynomial of degree $n-2d$. Note that $d$ is a positive integer in the case $p=2$.
When $Q(x)\in \C_1$ or $Q(x)\in \C_2$, and hence $Q$ is a Boolean function, the polynomial $\gcd(x^n+1,A(x))$ is also self-reciprocal, and $A(x)$ 
can be written as
\[ A(x) = x^df(x)g(x) \ , \]
where $f$ is a self-reciprocal divisor of $x^n+1$ of degree $s$, and $g$ is a self-reciprocal polynomial with degree smaller than
$n-s$, satisfying $\gcd(g,(x^n+1)/f) = 1$. \\
When $p$ is odd, then $x^n-1 = (x-1)\psi(x)$, where $\psi(x) = 1+x+\cdots+x^{n-1}\in\F_p[x]$ is self-reciprocal.
Hence $\gcd(x^n-1,A(x)) = (x-1)^\epsilon f(x)$, $\epsilon\in\{0,1\}$, for a self-reciprocal divisor $f$ of $x^n-1$. Thus $A(x)$
can be written as
\[ A(x) = x^d(x-1)^\epsilon f(x)g(x) \ , \]
where $g$ satisfies $\gcd(g,(x^n-1)/((x-1)^\epsilon f)) = 1$. \\
Obviously the factorization of $x^n+1$ and $\psi(x)$ into self-reciprocal factors plays an important role. In accordance 
with \cite{mt,mrt}, for a prime power $q$, we call a self-reciprocal polynomial $f\in\F_q[x]$ {\it prime self-reciprocal} if 
\begin{itemize}
\item[(i)] $f$ is irreducible over $\F_q$, or
\item[(ii)] $f=ugg^*$, where $g$ is irreducible over $\F_q$, the polynomial $g^*\ne g$ is the reciprocal 
of $g$ and $u\in\F_q^*$ is a constant.
\end{itemize}
To analyse the factorization of $x^n+1$ and $(x^n-1)/(x-1)$ into prime self-reciprocal polynomials, we recall the canonical 
factorization of $x^n-1$ into irreducible polynomials. Since $x^n-1 = (x^m-1)^{p^v}$ if $n = mp^v$, 
$\gcd(m,p) = 1$, we can assume that $n$ and $p$ are relatively prime. 
Let $\alpha$ be a primitive $n$th root of unity in an extension field of $\F_p$, and let 
$C_j = \{jp^k \bmod n\;:\; k\in\N\}$ be the {\it cyclotomic coset} of $j$ modulo $n$ (relative to powers of $p$).
Then $x^n-1 \in\F_p[x]$ can be factorized into irreducible polynomials as
\[ x^n-1 = \prod_{t=1}^hf_t(x)\quad\mbox{with}\quad f_t(x) = \prod_{i\in C_{j_t}}(x-\alpha^i), \]
where $C_{j_1},\ldots,C_{j_h}$ are the distinct cyclotomic cosets modulo $n$. \\
It is observed in \cite{mt,mrt} that, when $p$ is odd, an irreducible factor $f_t(x) = \prod_{i\in C_{j_t}}(x-\alpha^i)$ of $x^n-1$,
different from $x-1$, is self-reciprocal if and only if $C_{j_t}$ contains with $i$, its additive 
inverse $-i$ modulo $n$. Otherwise there exists a cyclotomic coset $C_{-j_t}$, which consists of the additive inverses 
of the elements of $C_{j_t}$, and the polynomial $f_t^*(x) = \prod_{i\in C_{-j_t}}(x-\alpha^i)$, which is the reciprocal of 
$f_t$. In this case $f_tf_t^*$ is a prime self-reciprocal divisor of $x^n-1$. \\
Most of our results are expressed in terms of the degrees of the prime self-reciprocal
factors of $x^n-1$. We remark that by Lemma 2 in \cite{mt}, the cardinalities of the 
cyclotomic cosets modulo $n$, and the degrees of the prime self-reciprocal divisors of $x^n-1$ can be obtained directly 
from the factorization of $n$.

As explained above our aim is to determine the 
generating polynomials $\mathcal{G}_n(z)$ and $\mathcal{H}_n(z)$ for even $n$, and $\mathcal{G}_n^{(p)}(z)$ for $n = p^vm$, $v>0$.
This enables us to solve the problem of enumerating
 quadratic functions in the sets $\C_1$, $\C_2$ and $\mathcal{D}$ with prescribed $s$. 
We adapt the number theoretical approach used in \cite[Section V]{mrt}, by which the generating
polynomial $\mathcal{G}_n(z)$ is obtained to yield the number of $s$-plateaued quadratic functions in $\C_1$ for the
two cases; odd $n$ and $n=2m$, where $m$ is odd. We start by giving some definitions and lemmas.

For a (self-reciprocal) polynomial $f\in \mathbb{F}_{p}[x]$ we define 
\begin{align*}
& C(f):=\lbrace g\in \mathbb{F}_{p}[x]\; |\; g \text{ is self-reciprocal},\;\deg(g)\;\text{is even, and }\mathrm{deg}(g) <\mathrm{deg}(f)  \rbrace \ ,\\
& K(f):= \lbrace g\in C(f) \; |\; \mathrm{gcd}(g(x),f(x))=1 \rbrace  \ ,\;\;\text{and}\\
& \phi_{p}(f):=|K(f)| \ .
\end{align*} 
%
Let $f$ be a monic self-reciprocal polynomial in $\mathbb{F}_{p}[x]$ with even degree. Let $f=r_{1}^{e_{1}}\cdots r_{k}^{e_{k}}$ 
be the factorization of $f$ into distinct monic self-reciprocal polynomials all of {\it even degree}, i.e. either 
$r_1 = (x+1)^2$ and $r_j$, $2\le j\le k$, are prime self-reciprocal polynomials, or $r_j$ is prime self-reciprocal
for all $1\le j\le k$. Then we define the (following variant of the M\"obius) function $\mu_p$ as 
\begin{align*}
\mu_{p}(f):=\left\lbrace \begin{array}{ll}
(-1)^{k}& \text{ if } e_{1}=\cdots=e_{k}=1 ,  \\
0& \text{ otherwise.}
\end{array} \right.
\end{align*} 
As for the classical M\"obius function on the set of positive integers, we have
\[ \sum_{d|f}\mu_p(d) = \left\{\begin{array}{ll}
1 & \quad\mbox{if}\; f=1,\\
0 & \quad\mbox{otherwise,}
\end{array}\right. \]
where the summation is over all monic self-reciprocal divisors $d$ of $f$ of even degree.
\begin{remark}
Compare $\mu_p$ with the "M\"obius function" used in \cite{mrt}, which is 
defined on the set of all self-reciprocal polynomials.
Here considering the M\"obius function on the set of self-reciprocal polynomials of {\it even degree}, where the set of
the prime elements is the union of $\{(x+1)^2\}$ and the set of prime self-reciprocal polynomials of even degree, proved to be 
advantageous and has facilitated obtaining $\mathcal{G}_n(z)$, $\mathcal{H}_n(z)$ and $\mathcal{G}_n^{(p)}(z)$
in full generality. 
\end{remark}
The next lemma is Lemma 8 in \cite{mrt}, except that the condition on a self-reciprocal polynomial "not to be divisible by $x+1$"
is replaced by the condition that it is of "even degree". The proof is similar to the one in \cite{mrt} 
with this slight modification and hence
we omit it.
\begin{lemma}\cite[Lemma 8]{mrt}
\label{13-8}
Let $f\in \mathbb{F}_{p}[x]$ be a monic self-reciprocal polynomial whose degree is a positive even integer. Then 
\begin{align*}
\sum_{d|f}\phi_{p}(d)=p^{\frac{\mathrm{deg}(f)}{2}}-1 \ ,
\end{align*}
where the sum runs over all monic self-reciprocal divisors $d$ of $f$ of even degree.
\end{lemma}
The next lemma is again similar to Lemma 9 in \cite{mrt}, except that polynomials involved are of even degree. We give the proof for the convenience of the reader.
\begin{lemma}\cite[Lemma 9]{mrt}
\label{13-9}
Let $f,f_1,f_2\in \mathbb{F}_{p}[x]$ be monic self-reciprocal polynomials whose degrees are positive even integers. 
\begin{itemize}
\item [(i)] We have
\begin{align*}
\phi_{p}(f)=\sum_{d|f}\mu_{p}(d)p^{\frac{\mathrm{deg}(f)-\mathrm{deg}(d)}{2}} \ ,
\end{align*}
where the sum runs over all monic self-reciprocal divisors $d$ of $f$ of even degree.
\item [(ii)] If $\mathrm{gcd}(f_1,f_2)=1$, then 
\begin{align*}
\phi_{p}(f)=\phi_{p}(f_1)\phi_{p}(f_2) \ .
\end{align*}
\end{itemize}
\end{lemma}
{\it Proof.}
Taking the summation over all monic self-reciprocal divisors $d$ of $f$ of even degree, by Lemma \ref{13-8} we get
\begin{align*}
& \sum_{d|f}\mu_p(d)\left(p^{\frac{\deg(f)-\deg(d)}{2}}-1\right) =
\sum_{d|f}\mu_p\left(\frac{f}{d}\right)\left(p^{\frac{\deg(d)}{2}}-1\right)  \\
& =\sum_{d|f}\mu_p\left(\frac{f}{d}\right)\sum_{d_1|d}\phi_p(d_1) =
\sum_{d_1|f}\phi_p(d_1)\sum_{g|\frac{f}{d_1}}\mu_p(g) = \phi_p(f)\ ,
\end{align*}
where in the last step we use the fact that the inner sum is $1$ for $d_1 = f$ and $0$ otherwise. With
\begin{eqnarray*}
\phi_p(f) & = & \sum_{d|f}\mu_p(d)\left(p^{\frac{\deg(f)-\deg(d)}{2}}-1\right) \\
& = & \sum_{d|f}\mu_p(d)p^{\frac{\deg(f)-\deg(d)}{2}} - \sum_{d|f}\mu_p(d) =
\sum_{d|f}\mu_p(d)p^{\frac{\deg(f)-\deg(d)}{2}},
\end{eqnarray*}
we finish the proof for (i). \\
If $\gcd(f_1,f_2) = 1$, then $\mu_p(f_1f_2) = \mu_p(f_1)\mu_p(f_2)$.
Taking the summation over all monic self-reciprocal divisors $d_i$ of $f_i$
with $\deg(d_i)$ is even, $i=1,2$, we then have
\begin{align*}
\phi_p(f_1)\phi_p(f_2)& = \sum_{d_1|f_1}\mu_p(d_1)p^{\frac{\deg(f_1)-\deg(d_1)}{2}}
\sum_{d_2|f_2}\mu_p(d_2)p^{\frac{\deg(f_2)-\deg(d_2)}{2}}  \\
& = \sum_{d_1|f_1\atop d_2|f_2}\mu_p(d_1d_2)p^{\frac{\deg(f_1f_2)-\deg(d_1d_2)}{2}} =
\sum_{d|f}\mu_p(d)p^{\frac{\deg(f)-\deg(d)}{2}}\ ,
\end{align*}
which shows (ii). \hfill$\Box$\\[.5em]
For an integer $t$, we define $\mathcal{N}_{n}(f;t)$ by
\begin{align}
\label{Nft}
\mathcal{N}_{n}(f;t)=\left\lbrace \begin{array}{ll}
1 & \text{ if } t=0 \\
\sum_{d|f, \mathrm{deg}(d)=t} \phi_{p}(d)& \text{ if $t>0$ is even}  \\
0 & \text{ otherwise,} 
\end{array}  \right.
\end{align}   
where the above sum runs over all monic self-reciprocal divisors $d$ of $f$ of degree $t$. Then we define the generating 
function $\mathcal{G}_n(f;z)$ for the distribution of the values $\mathcal{N}_{n}(f;t)$ by
\begin{align*}
\mathcal{G}_n(f;z)=\sum \mathcal{N}_{n}(f;t)z^{t} \ .
\end{align*}
We note that $\mathcal{G}_n(f;z)$ is a polynomial by definition of $\mathcal{N}_{n}(f;t)$.

The next lemma and its proof resembles Lemma 10 in \cite{mrt} for $p=2$. 
\begin{lemma}\cite[Lemma 10]{mrt}
\label{Gmulti}
Let $f=f_1f_2$ for two self-reciprocal polynomials $f_1,f_2$ of even degree. If $\mathrm{gcd}(f_1,f_2)=1$, then 
\begin{align*}
\mathcal{G}_n(f;z)=\mathcal{G}_n(f_1;z)\mathcal{G}_n(f_2;z)\ .
\end{align*}
\end{lemma}  
{\it Proof.}
We have to show that for each $t\ge0$,
\[ \mathcal{N}_n(f;t) =\sum_{0\le t_1\le t}\mathcal{N}_n(f_1;t_1)\mathcal{N}_n(f_2;t-t_1) \ .\]
For $t=0$, $t>\deg(f)$ and $t$ odd this is trivial. Assume that $1\le t\le \deg(f)$ is even. 
Taking into account that $\phi_p(1) = 0$, by Lemma \ref{13-9}(ii), and the definition of $\mathcal{N}_n(f;t)$ we obtain
\begin{eqnarray*}
& & \sum_{0\le t_1\le t}\mathcal{N}_n(f_1;t_1)\mathcal{N}_n(f_2;t-t_1) =
\mathcal{N}_n(f_1;0)\mathcal{N}_n(f_2;t) + \mathcal{N}_n(f_1;t)\mathcal{N}_n(f_2;0) \\
& &  + \sum_{1< t_1< t-1\atop\text{even}}\sum_{d_1|f_1\atop\deg(d_1)=t_1}\phi_p(d_1)\sum_{d_2|f_2\atop\deg(d_2) = t-t_1}\phi_p(d_2) \\
& = & \mathcal{N}_n(f_2;t) + \mathcal{N}_n(f_1;t) +
\sum_{1< t_1< t-1\atop\text{even}}\sum_{d_1|f_1;\deg(d_1) = t_1\atop d_2|f_2;\deg(d_2) = t-t_1}\phi_p(d_1d_2) \\
& = & \mathcal{N}_n(f_2;t) + \mathcal{N}_n(f_1;t) +
\sum_{d|f\atop \deg(d) = t}\phi_p(d) - \sum_{d_1|f_1\atop \deg(d_1) = t}\phi_p(d_1) -
\sum_{d_2|f_2\atop \deg(d_2) = t}\phi_p(d_2) \\
& = & \mathcal{N}_n(f;t).
\end{eqnarray*}
\hfill$\Box$
\begin{lemma}
\label{simpleG}
Let $r$ be a prime self-reciprocal polynomial of even degree. Then
\[ \mathcal{G}_{n}(r^k;z) = 1+\sum_{j=1}^{k}(p^{\frac{\mathrm{deg}(r^{j})}{2}}-p^{\frac{\mathrm{deg}(r^{j})-\mathrm{deg}(r)}{2}})z^{j\mathrm{deg}(r)}. \]
For an even integer $k$ we have
\[ \mathcal{G}_{n}((x+1)^{k};z) = \mathcal{G}_{n}((x-1)^{k};z) = 1+\sum_{j=1}^{\frac{k}{2}}p^{j-1}(p-1)z^{2j}. \]
\end{lemma}
{\it Proof.}
If $r$ is a prime self-reciprocal polynomial of even degree, then
\begin{eqnarray*}
 \mathcal{G}_{n}(r^{k};z)   &=&   \sum_{t:\text{ even}}\mathcal{N}_{n}(r^{k};t)z^{t}
 = 1+\sum_{t:\text{ even},\; t>0}\left(\sum_{d|r^{k}; \; \mathrm{deg}(d)=t}\phi_p(d)\right)z^{t}\\
   &=&1+ \sum_{j= 1}^{k}\phi_p(r^{j})z^{j\mathrm{deg}(r)}=
 1+  \sum_{j= 1}^{k}\left(\sum_{d|r^{j}}\mu_{p}(d)p^{\frac{\mathrm{deg}(r^{j})-\mathrm{deg}(d)}{2}}\right)z^{j\mathrm{deg}(r)} \\
   &=& 1+\sum_{j=
   1}^{k}\left(p^{\frac{\mathrm{deg}(r^{j})}{2}}-p^{\frac{\mathrm{deg}(r^{j})-\mathrm{deg}(r)}{2}}\right)z^{j\mathrm{deg}(r)}.
\end{eqnarray*}
The expression for $\mathcal{G}_{n}(((x-1)^2)^{k/2};z)$ follows as a special case, and $\mathcal{G}_{n}(((x+1)^2)^{k/2};z)$
is determined in the same way since the sum in $(\ref{Nft})$ is over self-reciprocal polynomials $d$ of even degree dividing $f$, and
$\mathcal{N}(f;t) = 0$ for odd $t$.
\hfill$\Box$

\section{Distribution of the nonlinearity in $\C_1$ and $\C_2$}
\label{secB}

Our aim in this section is to determine both $\mathcal{G}_n(z)$ and $\mathcal{H}_n(z)$ 
for all (even) integers $n$. By Remark \ref{rot.symm}, this does not only enable us to completely describe the distribution of the nonlinearity 
for the set of idempotent quadratic Boolean functions, but also 
for the set of rotation symmetric quadratic Boolean functions.

%

We first express our counting functions $\mathcal{N}_{n}(s)$ and $\mathcal{M}_{n}(s)$ in terms of $\mathcal{N}(f;t)$.
\begin{proposition}
\label{N+M}
Let $n$ be even and let $\mathcal{N}_{n}(s)$ and $\mathcal{M}_{n}(s)$ be the number of $s$-plateaued quadratic functions in $\C_1$
and $\C_2$, respectively. Then
\begin{equation*}
\mathcal{N}_{n}(s)=\mathcal{N}_{n}\left(\frac{x^{n}+1}{(x+1)^{2}};n-s\right)\quad\mbox{and}\quad
\mathcal{M}_{n}(s)=\mathcal{N}_{n}(x^{n}+1;n-s) \ .
\end{equation*}
\end{proposition}
{\it Proof.} The statement is clear when $n-s$ is zero or odd. Suppose $n-s > 0$ is even.
First we consider quadratic functions $Q \in \C_1$. For the corresponding associate polynomial $A(x)$
we have
$\mathrm{gcd}(A(x),x^{n}+1)=(x+1)^{2}f_{1}(x)$ for some self-reciprocal divisor $f_{1}$ of $(x^n+1)/(x^2+1)$
of degree $s-2$, i.e.
\begin{equation*}
    A(x)=x^c(x+1)^{2}f_{1}(x)g(x)
\end{equation*}
for an integer $c\ge 1$ and a self-reciprocal polynomial $g$ of even degree less than
$n-s$, which is relatively prime to
$d(x)=\frac{x^{n}+1}{(x+1)^{2}f_{1}(x)}$. In other words, $g$ is any of the $\phi_2(d)$
polynomials in $K(d)$.
To determine the number $\mathcal{N}_n(s)$ we consider all
divisors $(x+1)^2f_1(x)$ of $x^n+1$ of degree $s$, or equivalently, all divisors $d(x)$ of $(x^n+1)/(x^2+1)$
of degree $n-s$. Hence we obtain $\mathcal{N}_n(s)$ as
\begin{equation*}
\mathcal{N}_n(s) = \sum_{d|\frac{x^{n}+1}{(x+1)^{2}} \text{ and
}\mathrm{deg}(d)=n-s}\phi_{2}(d) \ .
\end{equation*}
If $Q\in\C_2$, then the associated polynomial $A(x)$ satisfies 
$\mathrm{gcd}(A(x),x^{n}+1)=f_{1}(x)$, where $f_{1}$ is a self-reciprocal polynomial of degree $s$, i.e.
\begin{equation*}
A(x)=x^cf_{1}(x)g(x)
\end{equation*}
for an integer $c\ge 1$ and a self-reciprocal polynomial $g$ of even degree less than
$n-s$ with  
$\mathrm{gcd}\left(g,(x^{n}+1)/f_{1}(x)\right)=1$. Therefore $g\in K(d)$.
As a consequence, the number of $s$-plateaued quadratic functions in $\C_2$ is
\begin{equation*}
\mathcal{M}_n(s) = \sum_{d|(x^{n}+1) \text{ and
}\mathrm{deg}(d)=n-s}\phi_{2}(d) \ ,
\end{equation*}
which finishes the proof. \hfill$\Box$ \\[.5em]
The following is the main theorem of this section. 
\begin{theorem}
\label{G}
Let $n = 2^vm$, $m$ odd, $v>0$, and let $x^n+1 = (x+1)^{2^v}r_1^{2^v}\cdots r_k^{2^v}$, where $r_1,\ldots,r_k$ are prime self-reciprocal 
polynomials of even degree. We set 
\begin{align*}
G_i(z):=1+\sum_{j=1}^{2^v}\left(2^{\frac{j\deg(r_i)}{2}}-2^{\frac{(j-1)\deg(r_i)}{2}}\right)z^{j\deg(r_i)} \ .
\end{align*}
Then the generating polynomial $\mathcal{G}_n(z) = \sum_{t=0}^n\mathcal{N}_n(n-t)z^t$
is given by
\[ \mathcal{G}_n(z) = \left(1+\sum_{j=1}^{2^{v-1}-1}2^{j-1}z^{2j}\right)\prod_{i=1}^k G_i(z)
\ , \]
and the generating polynomial $\mathcal{H}_n(z) = \sum_{t=0}^n\mathcal{M}_n(n-t)z^t$ is given by
\[ \mathcal{H}_n(z) = \left(1+\sum_{j=1}^{2^{v-1}}2^{j-1}z^{2j}\right)\prod_{i=1}^k G_i(z) \ . \]
\end{theorem}
{\it Proof.}
By Proposition \ref{N+M} and Lemma \ref{Gmulti} we have
\[ \mathcal{G}_n(z) = \mathcal{G}_{n}(\frac{x^n+1}{x^2+1};z)=
\mathcal{G}_{n}((x+1)^{2(2^{v-1}-1)};z)\prod_{i=1}^{k}\mathcal{G}_{n}(r_{i}^{2^{v}};z)\ , \]
and
\[ \mathcal{H}_n(z) = \mathcal{G}_{n}(x^n+1;z)=
\mathcal{G}_{n}((x+1)^{2(2^{v-1})};z)\prod_{i=1}^{k}\mathcal{G}_{n}(r_{i}^{2^{v}};z)\ . \]
For a prime self-reciprocal polynomial $r$ of even degree, Lemma \ref{simpleG} for $p=2$ gives
\begin{align*}
& \mathcal{G}_{n}(r^{2^{v}};z) =  
1+ \sum_{j=1}^{2^{v}}(2^{\frac{j\mathrm{deg}(r)}{2}}-2^{\frac{(j-1)\mathrm{deg}(r)}{2}})z^{j\mathrm{deg}(r)},\:\mbox{and} \\
& \mathcal{G}_{n}((x+1)^{2e};z) = 1+\sum_{j=1}^e2^{j-1}z^{2j}.
\end{align*}
Combining those formulas yields the assertion. \hfill$\Box$
\begin{remark}
Putting $v=1$ and $m > 1$ 
in Theorem \ref{G}, one obtains 
$\mathcal{G}_n(z)$ in 
Theorem 5(ii) of \cite{mrt}. Note that the Theorem 5(ii) in \cite{mrt} contains an
additional factor $2$, since a quadratic function there may also have
a linear term. Similarly the expression for $\mathcal{G}_n(z)$ with $m=1$ gives Theorem 6 in \cite{mt}.
\end{remark}
%
%
As a corollary of Theorem \ref{G} we obtain the number $\mathcal{M}_{n}(0)$ of bent functions in the set $\C_2$ as the coefficient of $z^n$
in $\mathcal{H}_n(z)$ for arbitrary (even) integers $n$. This complements the results of \cite{hf,yg}, where $\mathcal{M}_{n}(0)$ has 
been presented for the special cases $n = 2^vp^r$, where $p$ is a prime such that the order of $2$ modulo $p$ is $p-1$ or
$(p-1)/2$. 
%
\begin{corollary}
Let $n = 2^vm$, $m$ odd, $v>0$, and let $x^n+1 = (x+1)^{2^v}r_1^{2^v}\ldots r_k^{2^v}$, where $r_1,\cdots,r_k$ are prime self-reciprocal 
polynomials of even degree. Then the number of bent functions in $\C_2$ is
\[ \mathcal{M}_n(0) = 2^{2^{v-1}}\prod_{i=1}^k
\left(2^{\frac{2^v\deg(r_i)}{2}}-2^{\frac{(2^v-1)\deg(r_i)}{2}}\right)\ , \]
which is also the number of rotation symmetric quadratic bent functions in $n$ variables.
\end{corollary}
%
%
Similarly one may obtain $\mathcal{M}_n(s)$ and $\mathcal{N}_n(s)$ for other small values of $s$. 
The number of semi-bent function in $\C_1$ for even $n$, and in $\C_2$ is presented in the next corollary.
Note that the number of semi-bent functions in $\mathcal{C}_1$ when $n$ is odd is given in \cite[Corollary 7]{mrt}.
\begin{corollary}
Let $n = 2^vm$, $m$ odd, $v>0$, and let $x^n+1 = (x+1)^{2^v}r_1^{2^v}\ldots r_k^{2^v}$, where $r_1,\cdots,r_k$ are prime self-reciprocal 
polynomials of even degree. 
The number of semi-bent functions in $\C_1$ is
\[ \mathcal{N}_n(2) = 2^{2^{v-1}-2}\prod_{i=1}^k\left(2^{\frac{2^v\deg(r_i)}{2}}-2^{\frac{(2^v-1)\deg(r_i)}{2}}\right). \]
The number of semi-bent functions in $\C_2$ is
\begin{itemize}
\item[-] $\mathcal{M}_n(2) = \mathcal{N}_n(2)$ if $3$ does not divide $n$,
\item[-] $\mathcal{M}_n(2) = \mathcal{N}_n(2) + 2^{2^{v-1}-1}2^{2^v-2}\prod_{r_i\ne x^2+x+1}^k\left(2^{\frac{2^v\deg(r_i)}{2}}-2^{\frac{(2^v-1)\deg(r_i)}{2}}\right)$ if $3$ divides $n$.
\end{itemize}
\end{corollary}
{\it Proof.}
The corollary follows from Theorem \ref{G} with the observation that $x^2+x+1$ divides $x^n-1$ if and only if $3$ divides $n$.
\hfill$\Box$ 
\begin{remark}
In \cite[Theorem 12]{kkos}, for $n\equiv 0 \bmod 3$ an expression for the number $\mathcal{K}$ of semi-bent functions in $\C_2\setminus C_1$ 
is given, which involves a M\"obius function on the set of monic self-reciprocal polynomials. Our formula
\[ \mathcal{K} = 2^{2^{v-1}-1}2^{2^v-2}\prod_{r_i\ne x^2+x+1}^k\left(2^{\frac{2^v\deg(r_i)}{2}}-2^{\frac{(2^v-1)\deg(r_i)}{2}}\right) \]
is more explicit.
\end{remark}
\begin{remark}
To completely describe the nonlinearity distribution in $\mathcal{C}_1$
and $\mathcal{C}_2$ it is inevitable to consider the generating polynomials. Otherwise, in order to determine $\mathcal{N}_n(s)$ or $\mathcal{M}_n(s)$ for a specific $s$,
one would first have to find {\it all} possible ways of expressing $s$ as a sum of degrees of polynomials in the
prime self-reciprocal factorization of $x^n+1$, which for general $n$ is illusive.
\end{remark}

\section{Weight distribution of subcodes of second order Reed-Muller codes}
\label{secRM}

Let $\mathcal{Q}$ be a set of quadratic functions, which do not contain linear or constant terms. Assume that $\mathcal{Q}$
is closed under addition. Denote the set 
of affine functions from $\F_{2^n}$ to $\F_2$ by $\mathcal{A} = \{\Tr(bx)+c\; :\; b\in\F_{2^n},\; c\in\F_2\}$. Then the set
\[ \mathcal{Q} \oplus \mathcal{A} = \{Q(x)+l(x)\;:\; Q\in\mathcal{Q},\; l\in\mathcal{A}\} \]
gives rise to a linear subcode $\bar{R}_\mathcal{Q}$ of the second order Reed-Muller code $R(2,n)$, which 
contains the first order Reed-Muller code $R(1,n)$ as a subcode. Clearly, we can write $\mathcal{Q} \oplus \mathcal{A}$ 
as the union $\mathcal{Q} \oplus \mathcal{A} = \cup_{Q\in\mathcal{Q}} Q+\mathcal{A}$ of (disjoint) cosets of $\mathcal{A}$.
To obtain the weight distribution of the code $\bar{R}_\mathcal{Q}$, it is sufficient to know the weight distribution for 
each of these cosets.

It can be seen easily that the weight of the codeword $c_Q$ of a (quadratic) function $Q$ 
can be expressed in terms of the Walsh transform as
\[ wt(c_Q)=2^{n-1}-\frac{1}{2}\widehat{Q}(0) \ .\]
For a quadratic function $Q$
we define $Q_{b,c}(x) = Q(x) + \Tr(bx+c)$. 
Using $\widehat{Q_{b,c}}(0)=(-1)^{\Tr(c)}\widehat{Q}(b)$ one can show that the weight distribution of the coset $Q+\mathcal{A}$
for an $s$-plateaued quadratic function $Q$ is as follows. There are 
\begin{itemize}
\item[-] $2^{n-s}$ codewords of weight $2^{n-1}+2^{\frac{n+s}{2}-1}$,
\item[-] $2^{n-s}$ codewords of weight $2^{n-1}-2^{\frac{n+s}{2}-1}$, and 
\item[-] $2^{n+1}-2^{n-s+1}$ codewords of weight $2^{n-1}$.
\end{itemize}
Hence, if one knows the number of $s$-plateaued quadratic functions in $\mathcal{Q}$ for every $s$, one can determine
the weight distribution of $\bar{R}_\mathcal{Q}$. \\
If $\mathcal{Q}$ is the set of all quadratic functions, then $\bar{R}_\mathcal{Q} = R(2,n)$. 
The weight distribution of $R(2,n)$ is completely described in \cite{sb} by explicit, quite involved formulas. 
Here we focus on the subcodes of $R(2,n)$, obtained from the sets $\C_1$ and $\C_2$, in other words from 
the set of idempotent quadratic functions. 
Putting $k=n-s$ (which is even), the observations above imply that the only weights that can occur are 
$2^{n-1}$ and $2^{n-1}\pm 2^{n-1-\frac{k}{2}}$, $0\le k \le n$. Moreover, codewords of the weights $2^{n-1}+2^{n-1-\frac{k}{2}}$
and $2^{n-1}-2^{n-1-\frac{k}{2}}$ appear the same number of times. Hence to describe the weight distribution of the codes 
$\bar{R}_\C$ we
may consider the polynomial $\mathcal{W}_\C(z) = \sum_{k=0}^nA^{\C}_kz^k$, where $A^{\C}_k$ is the 
number of codewords in $\bar{R}_\C$ of weight $2^{n-1}\pm 2^{n-1-\frac{k}{2}}$. Again by the above observations, 
$A^{\C_1}_k = \mathcal{N}_n(n-k)2^k$ and $A^{\C_2}_k = \mathcal{M}_n(n-k)2^k$. Consequently,
\begin{eqnarray}
\label{W=}
\nonumber
\mathcal{W}_{\C_1}(z) & = & \sum_{k=0}^nA^{\C_1}_kz^k = \sum_{k=0}^n\mathcal{N}_n(n-k)2^kz^k = \mathcal{G}_n(2z)\ , \\
\mathcal{W}_{\C_2}(z) & = & \sum_{k=0}^nA^{\C_2}_kz^k = \sum_{k=0}^n\mathcal{M}_n(n-k)2^kz^k = \mathcal{H}_n(2z)\ .
\end{eqnarray}
For the number $A^{\C_1}$ of codewords in $\bar{R}_{\C_1}$ of weight $2^{n-1}$ we have
\begin{eqnarray*}
A^{\C_1} & = & \sum_{k=0}^{n}{\mathcal{N}_{n}(n-k)(2^{n+1}-2^{k+1})}
=2^{n+1}\sum_{k=0}^{n}{\mathcal{N}_{n}(n-k)}-2\sum_{k=0}^{n}{\mathcal{N}_{n}(n-k)2^{k}} \\
& = & 2^{n+1}\mathcal{G}_{n}(1)-2\mathcal{G}_{n}(2) \ .
\end{eqnarray*}
Similarly, $A^{\C_2} = 2^{n+1}\mathcal{H}_{n}(1)-2\mathcal{H}_{n}(2)$. \\
The following theorem describes the weight distribution of the codes $\bar{R}_\C$.
%
\begin{theorem}
Let $n = 2^tm$, $m$ odd, and let $x^n+1 = (x+1)^{2^t}r_1^{2^t}\cdots r_l^{2^t}$ for prime self-reciprocal polynomials 
$r_1,\ldots,r_l$ of even degree. Then for even $n$ 
\[ \mathcal{W}_{\C_2}(z) = \sum_{k=0}^nA^{\C_2}_kz^k =
\left(1+\sum_{j=1}^{2^{t-1}}2^{3j-1}z^{2j}\right)\prod_{i=1}^l W_i(z)
 \ , \]
where 
\begin{align*}
W_i(z)=1+\sum_{j=1}^{2^t}\left(2^{\frac{3j\deg(r_i)}{2}}-2^{\frac{(3j-1)\deg(r_i)}{2}}\right)z^{j\deg(r_i)} \ ,
\end{align*}
and
\begin{eqnarray*}
A^{\C_2} & = & 2^{n+1+2^{t-1}}\prod_{i=1}^l\left(1+\sum_{j=1}^{2^t}\left(2^{\frac{j\deg(r_i)}{2}}-2^{\frac{(j-1)\deg(r_i)}{2}}\right)\right) \\
& & - \frac{2^{3(2^{t-1}+1)}+6}{7}\prod_{i=1}^l\left(1+\sum_{j=1}^{2^t}\left(
2^{\frac{3j\deg(r_i)}{2}}-2^{\frac{(3j-1)\deg(r_i)}{2}}\right)\right)\ .
\end{eqnarray*}
When $t=0$, i.e., $n$ is odd, we have
\begin{equation}
\label{WB} 
\mathcal{W}_{\C_1}(z) = \sum_{k=0}^nA^{\C_1}_kz^k = \prod_{i=1}^{l}\left[1+(2^{3\mathrm{deg}(r_{i})/2}-2^{\mathrm{deg}(r_{i})})z^{\mathrm{deg}(r_{i})}\right]\ ,\;\mbox{and} 
\end{equation}
\[ A^{\C_1} = 2^{\frac{3n+1}{2}} - 2\prod_{i=1}^{l}\left(1+(2^{3\mathrm{deg}(r_{i})/2}-2^{\mathrm{deg}(r_{i})})\right)\ . \]
\end{theorem}
{\it Proof.} By using $(\ref{W=})$, the formulas for $\mathcal{W}_{\C_1}(z)$ and $\mathcal{W}_{\C_2}(z)$ follow from
the generating function
$\mathcal{G}_n(z) = \prod_{i=1}^{l}\left[1+(2^{\mathrm{deg}(r_{i})/2}-1)z^{\mathrm{deg}(r_{i})}\right]$ when $n$ is odd 
(see Theorem 5(i) in \cite{mrt}) and Theorem \ref{G}. The formulas for $A^{\C_1}$ and $A^{\C_2}$ are obtained by 
expanding $2^{n+1}\mathcal{G}_{n}(1)-2\mathcal{G}_{n}(2)$ and $2^{n+1}\mathcal{H}_{n}(1)-2\mathcal{H}_{n}(2)$.
\hfill$\Box$
\begin{remark}
When $n$ is odd, the code $\bar{R}_{\C_1}$ has $2^{(3n+1)/2}$ codewords, i.e. $\dim(\bar{R}_{\C_1}) = (3n+1)/2$.
Observing that the coefficient of $z^{k}$ in $(\ref{WB})$ is not zero if and only if 
$k=\sum_{r_{i}\in\{r_{1},\ldots, r_{l}\}}\mathrm{deg}(r_{i})$,
for $r=\mathrm{min}\{\deg(r_{i})\}_{i=1}^{l}$ we conclude that $\bar{R}_{\C_1}$ is a $[2^{n},(3n+1)/2,2^{n-1}-2^{n-1-\frac{r}{2}} ]$ code.
.
\end{remark}

\section{Enumeration of $s$-plateaued quadratic functions for odd characteristic}

In this section we apply our method to quadratic functions in $\D$. 
Recall that results on $\D$ translate to results on the set of rotation 
symmetric functions from $\F_p^n$ to $\F_p$.
As we will see, determining the generating function in odd characteristic is considerably more involved.
We can restrict ourselves to the case $\gcd(n,p) > 1$ since the case $\gcd(n,p) = 1$ is dealt with in 
\cite{mrt} with a different method which employs discrete Fourier transform. We emphasize that the method of \cite{mrt}
is not applicable to the case $\gcd(n,p) > 1$.

Recall that 
to $Q(x) = \mathrm{Tr}_n \left(\sum_{i=0}^{\lfloor n/2\rfloor}a_{i}x^{p^{i}+1} \right) \in \D$, the associate is
\begin{equation*}
    A(x) = \sum_{i=0}^{\lfloor n/2\rfloor}(a_ix^i+a_ix^{n-i}) \ ,
\end{equation*}
which is a polynomial of degree at most $n$. 
We treat the cases of odd and even $n$ separately.

\subsection*{The case of odd $n$}

In this case the factorization of $x^n-1$ into prime self-reciprocal divisors is given by
\begin{equation*}
x^n-1=(x-1)^{p^{v}}r_{1}^{p^{v}}\cdots r_{k}^{p^{v}} \ .
\end{equation*}
Firstly we note that $x^n-1$ is not divisible by $x+1$ as $n$ is an odd
integer, and hence each $r_{i}$, $1\le i\le k$, is a prime self-reciprocal polynomial of even degree. 
Write $A(x)$ as
\begin{equation*}
    A(x) = \sum_{i=0}^{\lfloor
    n/2\rfloor}(a_ix^i+a_ix^{n-i})=x^{c}h(x) \ ,
\end{equation*}
where $h(x)$ is a self-reciprocal polynomial of degree $n-2c$. Since
$n$ is odd, also the degree of $h$ is odd, and hence $x+1$ is a factor of $A(x)$.
In fact, $x+1$ must appear as an odd power in the factorization of $A(x)$. In particular,
\begin{equation*}
    A(x) =x^{c}(x+1)\varrho(x)
\end{equation*}
for a self-reciprocal polynomial $\varrho(x)$ of even degree $n-2c-1$.
As a consequence,
\begin{equation*}
\mathrm{gcd}( A(x),x^n-1)=\mathrm{gcd}( \varrho(x),x^n-1)=
(x-1)^{\epsilon}f_{1}(x) \ ,
\end{equation*}
for a self-reciprocal polynomial $f_{1}$ of even degree, and some $\epsilon\in \{0,1\}$. That is,
\begin{equation*}
\varrho(x)=(x-1)^{\epsilon}f_{1}(x)g_{1}(x) \; \text{with} \;
\mathrm{gcd}\left(\frac{x^n-1}{(x-1)^{\epsilon}f_{1}(x)},g_{1}(x)\right)=1 \ .
\end{equation*}

We will frequently use the following observation, which immediately follows from the definition of $\phi_p$.
\begin{proposition}\label{d}
For a self-reciprocal polynomial $d$ of even degree we have
\begin{equation*}
\phi_{p}((x\pm 1)d)=\phi_{p}((x\pm 1)^2 d) \ .
\end{equation*}
\end{proposition}
%
%
%

We now represent the counting function $\mathcal{N}^{(p)}_{n}(s)$ in
terms of the function $\mathcal{N}_{n}(f;t)$ in $(\ref{Nft})$.
The arguments are more complicated than in the even characteristic case, for instance, 
here the distinction between odd and even $n$ is required.
As we see in the proposition below, for odd $n$ we also need to consider odd and 
even $s$ separately. In the next subsection,
where $n$ is even, we have to treat four cases.
\begin{proposition}\label{s-n-odd} Let $n$ be an odd integer such that the $p$-adic valuation of $n$ is $v_{p}(n)=v$. 
The number $\mathcal{N}^{(p)}_{n}(s)$ of quadratic $s$-plateaued functions in $\mathcal{D}$ is given by
{ \footnotesize \begin{equation*}
\mathcal{N}^{(p)}_{n}(s)=\left\{\begin{array}{ll}
                        \mathcal{N}_{n}(\frac{x^n-1}{(x-1)^{p^{v}}};n-s)    & \text{ if $s$ is odd,} \\
                        \mathcal{N}_{n}((x-1)(x^n-1);n-s+1) - \mathcal{N}_{n}(\frac{x^n-1}{(x-1)^{p^v}};n-s+1)     & \text{ if $s$ is even.}
                          \end{array}
 \right.
\end{equation*}}
\end{proposition}
\begin{proof}
Suppose that $\mathrm{gcd}( A(x),x^n-1)=(x-1)f_{1}(x) $, i.e.
  $\epsilon=1$. This holds if only if $s$ is odd. In this case, $\mathrm{deg}(g_{1})$ is odd and $g_{1}(x)$ is
  divisible by $(x-1)$ since $\varrho(x)$ is a self-reciprocal polynomial of
  even degree, i.e. $g_{1}(x)=(x-1)g(x)$ for some self-reciprocal polynomial $g$ of even degree. Furthermore we know that
  $\mathrm{gcd}(\frac{x^n-1}{(x-1)f_{1}(x)},g_{1}(x))=1$, and
  therefore $x-1$ is not a factor of $\frac{x^n-1}{(x-1)f_{1}(x)}$.
Consequently, $(x-1)^{p^{v}-1}$ must divide $f_{1}(x)$. Hence $h(x)$ can be expressed as 
  \begin{equation*}
 h(x)=(x+1)(x-1)^{p^v+1}f(x)g(x) \; \;\text{with} \;\;
\mathrm{gcd}\left(\frac{x^n-1}{(x-1)^{p^v}f(x)},g(x)\right)=1 \ ,
  \end{equation*}
where $g$ is a self-reciprocal polynomial of even degree smaller than $n-s$, and hence $g\in K\left( \frac{x^n-1}{(x-1)^{p^v}f(x)} \right)$. 
Furthermore, each divisor $d$ of $\frac{x^n-1}{(x-1)^{p^v}}$ of degree $n-s$ uniquely determines $f(x)$. As a result, there exist
\begin{equation*}
\sum_{d|\frac{x^n-1}{(x-1)^{p^{v}}};\;
\mathrm{deg}(d)=n-s}\phi_p(d)=\mathcal{N}_{n}\left(\frac{x^n-1}{(x-1)^{p^{v}}};n-s\right)
\end{equation*} 
$s$-plateaued quadratic functions in $\mathcal{D}$ if $s$ is odd.
\vspace{.2cm}
Now we consider the case of even $s$, i.e. $\epsilon=0$ and $\mathrm{gcd}( A(x),x^n-1)=f(x)$, where the multiplicity of $x-1$ in $f$ is even. 
As a result, $\frac{x^n-1}{f(x)}$ is divisible by $x-1$, which implies that $g(x)$ is not divisible by $x-1$. In this case,
  \begin{equation*}
h(x)=(x+1)f(x)g(x) \;\; \;\text{with} \;\;\;
\mathrm{gcd}\left(\frac{x^n-1}{f(x)},g(x)\right)=1 \ ,
\end{equation*}
i.e. $g\in K\left(\frac{x^n-1}{f(x)} \right)=K\left((x-1)\frac{x^n-1}{f(x)} \right)$ by Proposition \ref{d}. Furthermore for any self-reciprocal 
divisor $d$ of $x^n-1$ of degree $n-s-1$, the factor $(x-1)d$ uniquely determines $f$. Hence there are 
\begin{align*}
&\sum_{(x-1)d|x^n-1;\;
\mathrm{deg}(d)=n-s-1}\phi_p((x-1)d)\\
&= \sum_{(x-1)^2d|(x-1)(x^n-1);\;
\mathrm{deg}(d)=n-s-1}\phi_p((x-1)^{2}d)
\end{align*}
$s$-plateaued functions in $\mathcal{D}$ in the case of even $s$, where $d$ runs over all monic self-reciprocal divisors of $x^n-1$. 
In order to determine the last sum, we separate the set 
\begin{equation*}
S = \{d\;:\; d\;\text{is monic, self-reciprocal,}~ d|(x-1)(x^n-1), ~\deg(d) = n-s+1\} 
\end{equation*} 
into two disjoint subsets.
The set $\mathcal{S}_{1}$ consists of elements of $S$ which are divisible by $x-1$, and $\mathcal{S}_{2}$ is the complement, i.e. 
\begin{align*}
\mathcal{S}_{1} &=\{d\;:\; d = (x-1)^2k \in S, ~k ~\text{is self-reciprocal,}~ \deg(k)=n-s-1 \}, \\
\mathcal{S}_{2} &=\left\{d\;:\; d\in S,~ d|\frac{x^n-1}{(x-1)^{p^v}} \right\}.
\end{align*}
Then
\[ \sum_{d\in S_1}\phi_p(d) = \sum_{d\in S}\phi_p(d) - \sum_{d\in S_2}\phi_p(d) \ , \]
or equivalently
\begin{align*}
&\sum_{(x-1)^{2}d|(x-1)(x^n-1);\;
\mathrm{deg}(d)=n-s-1}\phi_p((x-1)^{2}d)\\
&=\sum_{d|(x-1)(x^n-1);\;
\mathrm{deg}(d)=n-s+1}\phi_p(d)-\sum_{d|\frac{x^n-1}{(x-1)^{p^v}};\;
\mathrm{deg}(d)=n-s+1}\phi_p(d)\\
&=\mathcal{N}_{n}((x-1)(x^n-1);n-s+1) - \mathcal{N}_{n}(\frac{x^n-1}{(x-1)^{p^v}};n-s+1) \ .
\end{align*}
\end{proof}

\begin{remark} Note that from the proof of Proposition \ref{s-n-odd}, we see that $s\geq p^{v}$ if $s$ is odd.
\end{remark}

\begin{theorem}\label{pro:n-odd} Let $n$ be an odd integer and let $x^n-1=(x-1)^{p^{v}}r_{1}^{p^{v}}\cdots r_{k}^{p^{v}}$ be the factorization of 
$x^n-1$ into distinct prime self-reciprocal polynomials. 
Then 
$\mathcal{G}^{(p)}_{n}=\sum_{t=0}^{n}\mathcal{N}^{(p)}_{n}(n-t)z^{t}$ is given by
\begin{align*}
\mathcal{G}^{(p)}_{n}(z)=
\left( 1+\sum_{j= 1}^{\frac{p^{v}+1}{2}}p^{j-1}(p-1)z^{2j-1} \right)\prod_{i=1}^{k}G_i(z)\ ,
\end{align*}
where 
\begin{align*}
G_i(z)= 1+\sum_{j=
   1}^{p^{v}}(p^{\frac{\mathrm{deg}(r_i^{j})}{2}}-p^{\frac{\mathrm{deg}(r_i^{j})-\mathrm{deg}(r_i)}{2}})z^{j\mathrm{deg}(r_i)} \ .
\end{align*}
\end{theorem}
\begin{proof}
We can express the generating function as
{\footnotesize \begin{eqnarray*}
 \mathcal{G}^{(p)}_{n}(z)  &=&  \sum_{t:\text{ even}}\mathcal{N}^{(p)}_{n}(n-t)z^{t}+\sum_{t: \text{ odd}}\mathcal{N}^{(p)}_{n}(n-t)z^{t}\\
   &=&  \sum_{t:\text{ even}} \mathcal{N}_{n}\left(\frac{x^n-1}{(x-1)^{p^{v}}};t\right)z^{t}+\sum_{t: \text{ odd}}\left( \mathcal{N}_{n}((x-1)(x^n-1);t+1) - 
   \mathcal{N}_{n}\left(\frac{x^n-1}{(x-1)^{p^v}};t+1\right)\right) z^{t} \\
     &=&  \sum_{t:\text{ even}} \mathcal{N}_{n}\left(\frac{x^n-1}{(x-1)^{p^{v}}};t\right)z^{t}+\frac{1}{z}\sum_{t: \text{ odd}}\left( \mathcal{N}_{n}((x-1)(x^n-1);t+1) - \mathcal{N}_{n}\left(\frac{x^n-1}{(x-1)^{p^v}};t+1\right)\right) z^{t+1} \\
   &=&  \sum_{t:\text{ even}} \mathcal{N}_{n}\left(\frac{x^n-1}{(x-1)^{p^{v}}};t\right)z^{t}
   +\frac{1}{z}\left[\sum_{t:\text{ even}}\left( \mathcal{N}_{n}((x-1)(x^n-1);t) - 
   \mathcal{N}_{n}\left(\frac{x^n-1}{(x-1)^{p^v}};t\right)\right) z^{t}\right]\\
   &=&\mathcal{G}_{n}\left(\frac{x^n-1}{(x-1)^{p^{v}}};z\right)+
   \frac{1}{z}\left[\mathcal{G}_{n}((x-1)(x^n-1);z) - \mathcal{G}_{n}\left(\frac{x^n-1}{(x-1)^{p^v}};z \right)\right]\\
   &=& \prod_{i=1}^{k}\mathcal{G}_{n}(r_{i}^{p^{v}};z)\left(1-\frac{1}{z}+ \frac{1}{z}\mathcal{G}_{n}((x-1)^{p^{v}+1};z)\right).
\end{eqnarray*}}
Note that in the fourth equality 
we added and subtracted $\mathcal{N}_{n}((x-1)(x^n-1);0)=\mathcal{N}_{n}\left(\frac{x^n-1}{(x-1)^{p^v}};0\right)=1$. 
In the last equality we used Lemma \ref{Gmulti}. We then obtain the claimed formula by Lemma \ref{simpleG}.
\end{proof}


\subsection*{The case of even $n$}

In this case the factorization of $x^n-1$ is given by
\begin{equation*}
x^n-1=(x-1)^{p^{v}}(x+1)^{p^{v}}r_{1}^{p^{v}}\cdots r_{k}^{p^{v}} \
,
\end{equation*}
for some distinct prime self-reciprocal polynomials $r_i$, $1\le i\le k$, of even
degree. Now
\begin{equation*}
    A(x) = \sum_{i=0}^{\
    n/2}(a_ix^i+a_ix^{n-i})=x^{c}h(x) \ ,
\end{equation*}
where $h(x)$ is self-reciprocal, and of even degree $n-2c$.
For the $\gcd(A(x),x^n-1)$ we have
\begin{equation}
\label{epsde}
\mathrm{gcd}( A(x),x^n-1)=\mathrm{gcd}(h(x),x^n-1)=
(x-1)^{\epsilon}(x+1)^{\delta}f_{1}(x) \ ,
\end{equation}
for a self-reciprocal polynomial $f_{1}$ of even degree, and elements $\epsilon, \delta \in \{0,1\}$. 
Consequently,
\begin{equation*}
h(x)=(x-1)^{\epsilon}(x+1)^{\delta}f_{1}(x)g_1(x), \; \text{with} \;
\mathrm{gcd}\left(\frac{x^n-1}{(x-1)^{\epsilon}(x+1)^{\delta}f_{1}(x)},g_1(x)\right)=1
\ .
\end{equation*}
We now have to distinguish four cases depending on the values of $\epsilon$ and $\delta$ in $(\ref{epsde})$.
\begin{itemize}
  \item [(i)] Let $\epsilon=\delta=1$, i.e., $\mathrm{gcd}( A(x),x^n-1)=\mathrm{gcd}(h(x),x^n-1)=
(x-1)(x+1)f_{1}(x)$. In this case, $s$ is even and the
self-reciprocal polynomial $h$ of even degree is given by
\begin{align*}
&  h(x)=(x-1)(x+1)f_{1}(x)(x-1)(x+1)g(x) \\
&\quad \quad \quad \quad\quad \quad\text{with } \mathrm{gcd}\left((x-1)(x+1)g(x),\frac{x^n-1}{(x-1)(x+1)f_{1}}\right)=1 \ ,
\end{align*}
where $g(x)$ is a self-reciprocal polynomial of degree $\leq n-s-2$.
As a result, $f_{1}(x)=(x-1)^{p^{v}-1}(x+1)^{p^{v}-1}f(x)$ for a
self-reciprocal polynomial $f$ of even degree (which of course is a product of some $r_{i}$'s). 
As a consequence,
\begin{equation*}
h(x)=(x-1)^{p^{v}+1}(x+1)^{p^{v}+1}f(x)g(x) \ ,
\end{equation*}
for some self-reciprocal polynomial $g$. Note that $\mathrm{deg}(g)$ is an even integer smaller than $n-s$ and
$\mathrm{gcd}\left(g,\frac{x^n-1}{(x-1)^{p^{v}}(x+1)^{p^{v}}f(x)}\right)=1$.
Recalling that $\mathrm{deg}\left(\frac{x^n-1}{(x-1)^{p^{v}}(x+1)^{p^{v}}f(x)}\right)=n-s$,
we see that $g\in K\left(\frac{x^n-1}{(x-1)^{p^{v}}(x+1)^{p^{v}}f(x)}\right)$.
Furthermore each divisor $d$ of $\frac{x^n-1}{(x-1)^{p^{v}}(x+1)^{p^{v}}}$ of degree $n-s$ uniquely
determines $f$, and hence there exist
\begin{align*}
\sum_{d|\frac{x^n-1}{(x-1)^{p^{v}}(x+1)^{p^{v}}};\;
\mathrm{deg}(d)=n-s}\phi_p(d)
\end{align*}
such polynomials $h$.
\item [(ii)] Let $\epsilon=1$ and $\delta=0$, i.e., $\mathrm{gcd}( A(x),x^n-1)=\mathrm{gcd}(h(x),x^n-1)=
(x-1)f_{1}(x)$. In this case, $s$ is odd and the self-reciprocal polynomial $h$ of even degree is given by
\begin{equation*}
    h(x)=(x-1)f_{1}(x)(x-1)g(x)\;\; \text{with} \;\;
    \mathrm{gcd}\left((x-1)g(x),\frac{x^n-1}{(x-1)f_{1}}\right)=1 \ ,
\end{equation*}
for some self-reciprocal polynomial $g$,where $deg(g)$ is even, and smaller than $n-s$. Consequently, 
$f_{1}(x)=(x-1)^{p^{v}-1}f(x)$ for some self-reciprocal polynomial $f$ of even degree, and therefore
\begin{equation*}
h(x)=(x-1)^{p^{v}+1}f(x)g(x) \ ,
\end{equation*}
where $\mathrm{gcd}\left(g,\frac{x^n-1}{(x-1)^{p^{v}}f(x)}\right)=1$. 
Since $\frac{x^n-1}{(x-1)f_{1}}$ contains the factor $x+1$, the polynomial $g$ is not divisible by $x+1$.
In particular, $g \in K((x+1)d)$, where $d=\frac{x^n-1}{(x-1)^{p^{v}}(x+1)f(x)}$ is a
self-reciprocal polynomial of even degree $n-s-1$. As above,
each self-reciprocal divisor $d$ of $\frac{x^n-1}{(x-1)^{p^{v}}(x+1)}$ of degree
$n-s-1$ uniquely determines $f$. Consequently there exist
\begin{align*}
\sum_{d|\frac{x^n-1}{(x-1)^{p^{v}}(x+1)};\;
\mathrm{deg}(d)=n-s-1}\phi_p((x+1)d)
\end{align*}
such polynomials $h$.


\item [(iii)] The case $\epsilon=0$ and $\delta=1$ is very similar. With the same argument as in (ii) we get
\begin{equation}
\label{hiii}
    h(x)=(x+1)f_{1}(x)(x+1)g(x)\;\; \text{with} \;\;
    \mathrm{gcd}\left((x+1)g(x),\frac{x^n-1}{(x+1)f_{1}}\right)=1 \ ,
\end{equation}
for some self-reciprocal polynomial $g$, where $deg(g)$ is even, smaller than $n-s$, and 
\begin{align*}
\sum_{d|\frac{x^n-1}{(x+1)^{p^{v}}(x-1)};\;
\mathrm{deg}(d)=n-s-1}\phi_p((x-1)d)
\end{align*}
is the number of polynomials $h$ of the form $(\ref{hiii})$.
\item [(iv)] Let $\epsilon=\delta=0$, i.e., $\mathrm{gcd}( A(x),x^n-1)=\mathrm{gcd}(h(x),x^n-1)=
f(x)$ for some self-reciprocal polynomial $f$ of even degree. In this case
we see that the self-reciprocal polynomial $h$ is given by
\begin{equation*}
    h(x)=f(x)g(x)\; \text{with} \;
    \mathrm{gcd}\left(g(x),\frac{x^n-1}{f(x)}\right)=1 \ ,
\end{equation*}
where $g(x)$ is a self-reciprocal polynomial of degree $\leq n-s$. 
Note that $\frac{x^n-1}{f(x)}$ is divisible by both $x-1$ and $x+1$, so $g(x)$ is
not divisible by neither. Therefore
\begin{equation}
\label{hiiii}
    h(x)=f(x)g(x)\; \text{with} \;
    \mathrm{gcd}\left(g(x),(x-1)(x+1)\frac{x^n-1}{f(x)}\right)=1 \ ,
\end{equation}
where $g(x)$ is a self-reciprocal polynomial of degree smaller than $n-s+2$.
As for each self-reciprocal divisor $d$ of $x^n-1$ of degree $n-s$, $f$ is uniquely determined, we can again express
the number of polynomials $h$ of the form $(\ref{hiiii})$ in terms of $\phi_p$ as
\begin{align*}
\sum_{(x-1)(x+1)d|(x-1)(x+1)(x^n-1);\;
\mathrm{deg}(d)=n-s}\phi_p((x-1)(x+1)d).
\end{align*}
\end{itemize}

\begin{remark}
From the above observations we conclude that again $s\geq p^{v}$ if $s$ is odd.
\end{remark}

\begin{theorem}
\label{pro:n-even}
Let $n$ be an even integer and let $x^n-1=(x-1)^{p^{v}}(x+1)^{p^{v}}r_{1}^{p^{v}}\cdots r_{k}^{p^{v}}$ be the factorization of $x^n-1$ 
into distinct prime self-reciprocal polynomials.
Then the generating function $\mathcal{G}^{(p)}_{n}$ is given by
\begin{align*}
\mathcal{G}^{(p)}_{n}(z)=
\left( 1+\sum_{j= 1}^{\frac{p^{v}+1}{2}}p^{j-1}(p-1)z^{2j-1} \right)^2\prod_{i=1}^{k}G_i(z) \ ,
\end{align*}
where 
\begin{align*}
G_i(z)=  1+\sum_{j=
   1}^{p^{v}}(p^{\frac{\mathrm{deg}(r_i^{j})}{2}}-p^{\frac{\mathrm{deg}(r_i^{j})-\mathrm{deg}(r_i)}{2}})z^{j\mathrm{deg}(r_i)} \ .
\end{align*}
\end{theorem}

\begin{proof}
First let $s$ be odd, which applies in the cases (ii) and (iii) above. The discussion of these cases imply for 
$\mathcal{N}^{(p)}_{n}(s)$ that 
{\scriptsize \begin{eqnarray}
\label{I+II}
   \mathcal{N}^{(p)}_{n}(s) &=& \sum_{d|\frac{x^n-1}{(x-1)^{p^{v}}(x+1)};\;
\mathrm{deg}(d)=n-s-1}\phi_p((x+1)d)+\sum_{d|\frac{x^n-1}{(x+1)^{p^{v}}(x-1)};\;
\mathrm{deg}(d)=n-s-1}\phi_p((x-1)d)\nonumber\\
 &=& \sum_{d|\frac{x^n-1}{(x-1)^{p^{v}}(x+1)};\;
\mathrm{deg}(d)=n-s-1}\phi_p((x+1)^{2}d)+\sum_{d|\frac{x^n-1}{(x+1)^{p^{v}}(x-1)};\;
\mathrm{deg}(d)=n-s-1}\phi_p((x-1)^{2}d) \ . \nonumber\\
& &
\end{eqnarray}}
To determine the sum $(\ref{I+II})$ we separate the set 
\begin{equation*}
S=\left \{h\;:\; h\;\text{is monic, self-reciprocal,}~ h|\frac{x^n-1}{(x-1)^{p^v}},~ 
\deg(h) = n-s+1\right\}
\end{equation*}
into the two disjoint subsets
\begin{eqnarray*}
  S_{1} &=&\{h\;:\; h\in S \; \text{and} \; \mathrm{gcd}(h, x+1)=1\},\;\mbox{and}  \\
    S_{2} &=&\{h\;:\; h\in S \; \text{and} \; (x+1)|h\}.
\end{eqnarray*}
Note that being self-reciprocal and of even degree, the polynomials $h$ in the set $S_2$
are divisible by $(x+1)^{2}$. Then for the first sum in $(\ref{I+II})$ we obtain
\begin{eqnarray*}
& & \sum_{d|\frac{x^n-1}{(x-1)^{p^{v}}(x+1)};\;
\mathrm{deg}(d)=n-s-1}\phi_p((x+1)^{2}d)=  \sum_{h\in
S_{2}}\phi_p(h)\\
& &=\sum_{h|\frac{(x+1)(x^n-1)}{(x-1)^{p^{v}}};\;
  \mathrm{deg}(h)=n-s+1}\phi_p(h)-\sum_{h\in S_{1}}\phi_p(h)\\
   & &=\sum_{h|\frac{(x+1)(x^n-1)}{(x-1)^{p^{v}}};\;
  \mathrm{deg}(h)=n-s+1}\phi_p(h)-\sum_{h|\frac{x^n-1}{(x-1)^{p^{v}}(x+1)^{p^{v}}};\;
  \mathrm{deg}(h)=n-s+1}\phi_p(h)\\
  & &={\mathcal{N}}_{n}\left(\frac{(x+1)(x^n-1)}{(x-1)^{p^{v}}};n-s+1\right)-{\mathcal{N}}_{n}\left(\frac{x^n-1}{(x-1)^{p^{v}}(x+1)^{p^{v}}};n-s+1\right).
\end{eqnarray*}
In a similar way, for the second sum in $(\ref{I+II})$ we get
\begin{align*}
&\sum_{d|\frac{x^n-1}{(x+1)^{p^{v}}(x-1)};\;
\mathrm{deg}(d)=n-s-1}\phi_p((x-1)^{2}d) \\
& {\mathcal{N}}_{n}\left(\frac{(x-1)(x^n-1)}{(x+1)^{p^{v}}};n-s+1\right)-{\mathcal{N}}_{n}\left(\frac{x^n-1}{(x-1)^{p^{v}}(x+1)^{p^{v}}};n-s+1\right).
\end{align*}
Combining, we obtain
\begin{eqnarray*}
  \mathcal{N}^{(p)}_{n}(s) &=&{\mathcal{N}}_{n}\left(\frac{(x+1)(x^n-1)}{(x-1)^{p^{v}}};n-s+1\right)
  +{\mathcal{N}}_{n}\left(\frac{(x-1)(x^n-1)}{(x+1)^{p^{v}}};n-s+1\right)\\
   &&-2{\mathcal{N}}_{n}\left(\frac{x^n-1}{(x-1)^{p^{v}}(x+1)^{p^{v}}};n-s+1\right)
   \ 
\end{eqnarray*}
for an odd integer $s$.
\vspace{.2cm}
Now we consider the case of even $s$, which corresponds to (i) and (iv). By the above observations, 
\begin{eqnarray*}
  \mathcal{N}^{(p)}_{n}(s) &=& \sum_{d|\frac{x^n-1}{(x-1)^{p^{v}}(x+1)^{p^{v}}};\;
\mathrm{deg}(d)=n-s}\phi_p(d)\\&&+\sum_{(x-1)(x+1)d|(x-1)(x+1)(x^n-1);\;
\mathrm{deg}(d)=n-s}\phi_p((x-1)(x+1)d) \\
   &=&\mathcal{N}_{n}\left(\frac{x^n-1}{(x-1)^{p^{v}}(x+1)^{p^{v}}};n-s\right)\\&&+\sum_{(x-1)(x+1)d|(x-1)(x+1)(x^n-1);\;
\mathrm{deg}(d)=n-s}\phi_p((x-1)(x+1)d).
\end{eqnarray*}
\vspace{.5cm}
We set $T:=\sum_{(x-1)(x+1)d|(x-1)(x+1)(x^n-1);\;\mathrm{deg}(d)=n-s}\phi_p((x-1)(x+1)d)$ and in order to determine $T$ we consider the set
\[ S = \{h(x)\;:\;h(x)\,\text{is self-reciprocal, }~ h|(x-1)(x+1)(x^{n}-1),~ \deg(h) = n-s+2\} \]
and its subsets
\begin{align*}
&S_{1} = \{h(x)\;:\; h\in S, \; \mathrm{gcd}(h, x-1)=1\}, \\
& S_{2} =\{h(x)\;:\; h\in S, \; \mathrm{gcd}(h, x+1)=1\}, \\
&S_{3} =\{h(x)\;:\; h\in S, \; \mathrm{gcd}(h, x^{2}-1)=1\}. 
\end{align*}
Note that since $h$ is a self-reciprocal polynomial of even degree, $h$ is divisible by $(x\mp 1)^2$ if it is divisible by $x\mp 1$. As a consequence,
\begin{align*}
T=\sum_{h|(x-1)(x+1)(x^n-1);\;
\mathrm{deg}(h)=n-s+2}\phi_p(h)-\sum_{h\in S_{1} }\phi_p(h)-\sum_{h\in S_{2} }\phi_p(h)+\sum_{h\in S_{3} }\phi_p(h).
\end{align*}
We also note that for $h$ in the set $S$, $h\in S_{1} $ if and only if $h|(x+1)\frac{x^n-1}{(x-1)^{p^{v}}}$, $h\in S_{2} $ if and only if
$h|(x-1)\frac{x^n-1}{(x+1)^{p^{v}}}$ and  $h\in S_{3}$ if and only if $h|\frac{x^n-1}{(x-1)^{p^{v}}(x+1)^{p^{v}}}$. Therefore we can express 
the sum $T$ as
\begin{align*}
T=\sum_{h|(x-1)(x+1)(x^n-1);\;
\mathrm{deg}(h)=n-s+2}\phi_p(h)-\sum_{h|(x+1)\frac{x^n-1}{(x-1)^{p^{v}}};\;
\mathrm{deg}(h)=n-s+2}\phi_p(h)\\
-\sum_{h|(x-1)\frac{x^n-1}{(x+1)^{p^{v}}};\;
\mathrm{deg}(h)=n-s+2 }\phi_p(h)
+\sum_{h|\frac{x^n-1}{(x-1)^{p^{v}}(x+1)^{p^{v}}};\;
\mathrm{deg}(h)=n-s+2  }\phi_p(h).
\end{align*}
Hence, for an even integer $s$
{ \scriptsize \begin{eqnarray*}
  \mathcal{N}^{(p)}_{n}(s) &=& \mathcal{N}_{n}\left(\frac{x^n-1}{(x-1)^{p^{v}}(x+1)^{p^{v}}};n-s\right)+\mathcal{N}_{n}\left((x-1)(x+1)(x^{n}-1);n-s+2\right) \\
 &&- \mathcal{N}_{n}\left(\frac{(x+1)(x^{n}-1)}{(x-1)^{p^{v}}};
  n-s+2\right) - \mathcal{N}_{n}\left(\frac{(x-1)(x^{n}-1)}{(x+1)^{p^{v}}}; n-s+2\right)\\ && + \mathcal{N}_{n}\left(\frac{x^n-1}{(x-1)^{p^{v}}(x+1)^{p^{v}}};n-s+2\right).
\end{eqnarray*}}
\vspace{.3cm}

From the above equalities 
we get
{\scriptsize \begin{align*}
\mathcal{G}_n^{(p)}(z) &= \sum_{t=0}^{n} \mathcal{N}^{(p)}_{n}(n-t)z^{t}=\sum_{t: \text{ odd}} \mathcal{N}^{(p)}_{n}(n-t)z^{t}+\sum_{t: \text{ even}} \mathcal{N}^{(p)}_{n}(n-t)z^{t} \\
&=\sum_{t: \text{ odd}}\left[\mathcal{N}_{n}\left(\frac{(x+1)(x^n-1)}{(x-1)^{p^{v}}};t+1\right)
  +\mathcal{N}_{n}\left(\frac{(x-1)(x^n-1)}{(x+1)^{p^{v}}};t+1\right)-2\mathcal{N}_{n}\left(\frac{x^n-1}{(x-1)^{p^{v}}(x+1)^{p^{v}}};t+1\right)\right]z^{t}\\
  &+\sum_{t: \text{ even}}\left[  \mathcal{N}_{n}\left(\frac{x^n-1}{(x-1)^{p^{v}}(x+1)^{p^{v}}};t\right)+\mathcal{N}_{n}\left((x-1)(x+1)(x^{n}-1);t+2\right)- \mathcal{N}_{n}\left(\frac{(x+1)(x^{n}-1)}{(x-1)^{p^{v}}};
  t+2\right)  \right.\\
  & \left. - \mathcal{N}_{n}\left(\frac{(x-1)(x^{n}-1)}{(x+1)^{p^{v}}}; t+2\right) 
  +\mathcal{N}_{n}\left(\frac{x^n-1}{(x-1)^{p^{v}}(x+1)^{p^{v}}};t+2\right)\right]z^t
     \end{align*}
  \begin{align*}
  &=\frac{1}{z}\sum_{t: \text{ odd}}\left[\mathcal{N}_{n}\left(\frac{(x+1)(x^n-1)}{(x-1)^{p^{v}}};t+1\right)
  +\mathcal{N}_{n}\left(\frac{(x-1)(x^n-1)}{(x+1)^{p^{v}}};t+1\right)-2\mathcal{N}_{n}\left(\frac{x^n-1}{(x-1)^{p^{v}}(x+1)^{p^{v}}};t+1\right)\right]z^{t+1}\\
  &+\sum_{t: \text{ even}} \mathcal{N}_{n}\left(\frac{x^n-1}{(x-1)^{p^{v}}(x+1)^{p^{v}}};t\right)z^t
  +\frac{1}{z^{2}}\sum_{t: \text{ even}}\Big[ \mathcal{N}_{n}\left((x-1)(x+1)(x^{n}-1);t+2\right) \\
  & - \mathcal{N}_{n}\left(\frac{(x+1)(x^{n}-1)}{(x-1)^{p^{v}}};
  t+2\right)- \mathcal{N}_{n}\left(\frac{(x-1)(x^{n}-1)}{(x+1)^{p^{v}}}; t+2\right) 
  +\mathcal{N}_{n}\left(\frac{x^n-1}{(x-1)^{p^{v}}(x+1)^{p^{v}}};t+2\right)\Big]z^{t+2}
   \end{align*}
  \begin{align*}
  &=\frac{1}{z}\sum_{t: \text{ even}}\left[\mathcal{N}_{n}\left(\frac{(x+1)(x^n-1)}{(x-1)^{p^{v}}};t\right)
  +\mathcal{N}_{n}\left(\frac{(x-1)(x^n-1)}{(x+1)^{p^{v}}};t\right)-2\mathcal{N}_{n}\left(\frac{x^n-1}{(x-1)^{p^{v}}(x+1)^{p^{v}}};t\right)\right]z^{t}\\
  &+\sum_{t: \text{ even}} \mathcal{N}_{n}\left(\frac{x^n-1}{(x-1)^{p^{v}}(x+1)^{p^{v}}};t\right)z^t
  +\frac{1}{z^{2}}\sum_{t: \text{ even}}\left[ \mathcal{N}_{n}\left((x-1)(x+1)(x^{n}-1);t\right)  \right.\\
  & \left. - \mathcal{N}_{n}\left(\frac{(x+1)(x^{n}-1)}{(x-1)^{p^{v}}};
  t\right)- \mathcal{N}_{n}\left(\frac{(x-1)(x^{n}-1)}{(x+1)^{p^{v}}}; t\right) 
  +\mathcal{N}_{n}\left(\frac{x^n-1}{(x-1)^{p^{v}}(x+1)^{p^{v}}};t\right)\right]z^{t}
    \end{align*}
  \begin{align*}
    &=\frac{1}{z}\left[\mathcal{G}_{n}\left(\frac{(x+1)(x^n-1)}{(x-1)^{p^{v}}};z\right)
  +\mathcal{G}_{n}\left(\frac{(x-1)(x^n-1)}{(x+1)^{p^{v}}};z\right)-2\mathcal{G}_{n}\left(\frac{x^n-1}{(x-1)^{p^{v}}(x+1)^{p^{v}}};z\right)\right]\\
  &+ \mathcal{G}_{n}\left(\frac{x^n-1}{(x-1)^{p^{v}}(x+1)^{p^{v}}};z\right)
  +\frac{1}{z^{2}}\Big[ \mathcal{G}_{n}\left((x-1)(x+1)(x^{n}-1);z\right) \\
  & - \mathcal{G}_{n}\left(\frac{(x+1)(x^{n}-1)}{(x-1)^{p^{v}}};
  z\right)- \mathcal{G}_{n}\left(\frac{(x-1)(x^{n}-1)}{(x+1)^{p^{v}}}; z\right) 
  +\mathcal{G}_{n}\left(\frac{x^n-1}{(x-1)^{p^{v}}(x+1)^{p^{v}}};z\right)\Big]
    \end{align*}
  \begin{align*}
 & =\prod_{i=1}^{k}\mathcal{G}_{n}(r_{i}^{p^{v}};z)\left( \frac{1}{z}\mathcal{G}_{n}((x+1)^{p^{v}+1};z)+\frac{1}{z}\mathcal{G}_{n}((x-1)^{p^{v}+1};z)-\frac{2}{z}+1\right.\\
 &\left.+\frac{1}{z^{2}}\mathcal{G}_{n}((x-1)^{p^{v}+1}(x+1)^{p^{v}+1};z)- \frac{1}{z^{2}}\mathcal{G}_{n}((x+1)^{p^{v}+1};z) -\frac{1}{z^{2}}\mathcal{G}_{n}((x-1)^{p^{v}+1};z)+\frac{1}{z^{2}}\right).
\end{align*}}
By Lemma \ref{simpleG} we get
\begin{align*}
& \mathcal{G}^{(p)}_{n}(z)=\prod_{i=1}^{k}\mathcal{G}_{n}(r_{i}^{p^{v}};z)\left( \frac{2}{z}\mathcal{G}_{n}((x-1)^{p^{v}+1};z)-\frac{2}{z}+1+\frac{1}{z^{2}}\mathcal{G}_{n}((x-1)^{p^{v}+1};z)^2\right.\\
 &\left.-\frac{2}{z^{2}}\mathcal{G}_{n}((x-1)^{p^{v}+1};z)+\frac{1}{z^{2}}\right).
\end{align*}
Putting
\begin{align*}
&G_{i}:=\mathcal{G}_{n}(r_{i}^{p^{v}};z)\ , \;\; \text{and} \;\;A:=\sum_{j= 1}^{\frac{p^{v}+1}{2}}p^{j-1}(p-1)z^{2j} =\mathcal{G}_{n}((x-1)^{p^{v}+1};z)-1
\end{align*}
this yields
\begin{align*}
&\mathcal{G}^{(p)}_{n}(z)&=&\prod_{i=1}^{k} G_{i} \left[\frac{2}{z} (1+A)-\frac{2}{z} +1+\frac{1}{z^{2}}(1+A)^2 -\frac{2}{z^{2}}(1+A)+\frac{1}{z^{2}}\right]\\
&&=&\left(1+\frac{A}{z}\right)^2 \prod_{i=1}^{k} G_{i}.
\end{align*}
Simplifying and using Lemma \ref{simpleG} we get the desired result.
\end{proof}

As an immediate corollary of Theorems \ref{pro:n-odd} and \ref{pro:n-even} we obtain the number of bent functions in the 
set $\mathcal{D}$. Note that the case of $\gcd(n,p)=1$ is covered in \cite[Corollary 7]{mrt}.
\begin{corollary}
Let the factorization of $x^n-1$ into distinct prime self-reciprocal polynomials be
$x^n-1=(x-1)^{p^{v}}r_{1}^{p^{v}}\cdots r_{k}^{p^{v}}$ when $n$ is odd, and 
$x^n-1=(x-1)^{p^{v}}(x+1)^{p^{v}}r_{1}^{p^{v}}\cdots r_{k}^{p^{v}}$ when $n$ is even.
Then the number of bent functions in $\mathcal{D}$ is
\[ \mathcal{N}_n(0) = (p-1)p^{\frac{p^v-1}{2}}\prod_{i=1}^k\left(p^{\frac{p^v\deg(r_i)}{2}}-p^{\frac{(p^v-1)\deg(r_i)}{2}}\right)\ , \]
if $n$ is odd, and
\[ \mathcal{N}_n(0) = 
((p-1)p^{\frac{p^v-1}{2}})^2\prod_{i=1}^k\left(p^{\frac{p^v\deg(r_i)}{2}}-p^{\frac{(p^v-1)\deg(r_i)}{2}}\right)\ , \]
if $n$ is even.
\end{corollary}

%

\newpage

\section{Appendix}

We give some examples of generating function for $p=3$. \\[.5em]
{\it Example $n=9\cdot 13$}: In this case, $x^{9\cdot 13}-1 = (x-1)^9r_1^9r_2^9$ for prime self-reciprocal polynomials $r_1,r_2$ both of degree $6$.
By Theorem \ref{pro:n-odd},
\[ \mathcal{G}_{9\cdot 13}^{(3)}(z) = (1+2\sum_{j=1}^53^{j-1}z^{2j-1})(1+26\sum_{j=1}^93^{3(j-1)}z^{6j})^2. \]
Expanding this polynomial we obtain
{\scriptsize \begin{align*}
& \mathcal{G}_{9\cdot 13}^{(3)}(z) = \\
& 1+2\,z+6\,{z}^{3}+18\,{z}^{5}+52\,{z}^{6}+158\,{z}^{7}+474\,{z}^{9}+
936\,{z}^{11}+2080\,{z}^{12}+6968\,{z}^{13}+20904\,{z}^{15}+37440\,{z}
^{17}\\
& +74412\,{z}^{18}+261144\,{z}^{19}+783432\,{z}^{21}+1339416\,{z}^{
23}+2501928\,{z}^{24}+9022104\,{z}^{25}+27066312\,{z}^{27}\\ & +45034704\,{
z}^{29}
+80857764\,{z}^{30}+296819640\,{z}^{31}+890458920\,{z}^{33}+
1455439752\,{z}^{35}+2542413744\,{z}^{36}\\ & +9451146744\,{z}^{37}+
28353440232\,{z}^{39}+45763447392\,{z}^{41}+78345032220\,{z}^{42}+
293980406616\,{z}^{43}\\
&+881941219848\,{z}^{45} +1410210579960\,{z}^{47}+
2377212120504\,{z}^{48}+8985055980888\,{z}^{49}+26955167942664\,{z}^{
51}\\ &+42789818169072\,{z}^{53}
+71255926018836\,{z}^{54}+270881306544888
\,{z}^{55}+812643919634664\,{z}^{57}\\ & +1282606668339048\,{z}^{59}+
1718301299950404\,{z}^{60}+7284422604917952\,{z}^{61}+
21853267814753856\,{z}^{63}\\ 
&+30929423399107272\,{z}^{65} +
41239231198809696\,{z}^{66}+175266732594941208\,{z}^{67}+
525800197784823624\,{z}^{69}\\ 
& +742306161578574528\,{z}^{71} +
974276837071879068\,{z}^{72}+4175472158879481720\,{z}^{73}+
12526416476638445160\,{z}^{75}\\
& +17536983067293823224\,{z}^{77}+
22547549657949201288\,{z}^{78}+97706048517779872248\,{z}^{79}\\
& +
293118145553339616744\,{z}^{81}+405855893843085623184\,{z}^{83}+
507319867303857028980\,{z}^{84}\\
&+2232207416136970927512\,{z}^{85} +
6696622248410912782536\,{z}^{87}+9131757611469426521640\,{z}^{89}\\
& +
10958109133763311825968\,{z}^{90}+49311491101934903216856\,{z}^{91}+
147934473305804709650568\,{z}^{93}\\
& +197245964407739612867424\,{z}^{95} +
221901709958707064475852\,{z}^{96}+1035541313140632967553976\,{z}^{97}\\
& +3106623939421898902661928\,{z}^{99} +3994230779256727160565336\,{z}^{
101}+3994230779256727160565336\,{z}^{102}\\
& +19971153896283635802826680\,
{z}^{103} +59913461688850907408480040\,{z}^{105}+
71896154026621088890176048\,{z}^{107}\\
& +53922115519965816667632036\,{z}^
{108} +323532693119794900005792216\,{z}^{109}\\
& +
970598079359384700017376648\,{z}^{111}+970598079359384700017376648\,{z
}^{113}\\
& +2911794238078154100052129944\,{z}^{115} +
8735382714234462300156389832\,{z}^{117}
\end{align*}}

\vspace{1em}
{\it Example $n=9\cdot 14$}: In this case, $x^{9\cdot 14}-1 = (x-1)^9(x+1)^9r_1^9r_2^9$ for prime self-reciprocal polynomials $r_1,r_2$ both of degree $6$.
By Theorem \ref{pro:n-even},
{ \scriptsize \begin{align*}
&\mathcal{G}_{9\cdot 14}^{(3)}(z) = (1+2\sum_{j=1}^53^{j-1}z^{2j-1})^2(1+26\sum_{j=1}^93^{3(j-1)}z^{6j})^2 = \\
&1+4\,z+4\,{z}^{2}+12\,{z}^{3}+24\,{z}^{4}+36\,{z}^{5}+160\,{z}^{6}+316
\,{z}^{7}+640\,{z}^{8}+948\,{z}^{9}+2868\,{z}^{10}+1872\,{z}^{11}+
11584\,{z}^{12}\\
&+13936\,{z}^{13}+39532\,{z}^{14}+41808\,{z}^{15}+151656
\,{z}^{16}+74880\,{z}^{17}+527472\,{z}^{18}+522288\,{z}^{19}+1651104\,
{z}^{20}\\
& +1566864\,{z}^{21}+6065280\,{z}^{22}+2678832\,{z}^{23}+
19990152\,{z}^{24}+18044208\,{z}^{25}+60349536\,{z}^{26}+54132624\,{z}
^{27}\\
& +216985392\,{z}^{28}+90069408\,{z}^{29}+694967364\,{z}^{30}+
593639280\,{z}^{31}+2055220128\,{z}^{32}+1780917840\,{z}^{33}\\
& +
7295622048\,{z}^{34}+2910879504\,{z}^{35}+22955416848\,{z}^{36}+
18902293488\,{z}^{37}+66987075168\,{z}^{38}\\
&+56706880464\,{z}^{39}+
235781239824\,{z}^{40}+91526894784\,{z}^{41}+732961301436\,{z}^{42}+
587960813232\,{z}^{43}\\
& +2119046585760\,{z}^{44}+1763882439696\,{z}^{45}
+7413678477504\,{z}^{46}+2820421159920\,{z}^{47}+22845411395352\,{z}^{
48}\\
& +17970111961776\,{z}^{49}+65594937833568\,{z}^{50}+53910335885328\,
{z}^{51}+228454113953520\,{z}^{52}\\
&+85579636338144\,{z}^{53}+
699323426602164\,{z}^{54}+541762613089776\,{z}^{55}+1997341681993632\,
{z}^{56}\\
&+1625287839269328\,{z}^{57}+6931950543389664\,{z}^{58}+
2565213336678096\,{z}^{59}+20712629060085924\,{z}^{60}\\
& +
14568845209835904\,{z}^{61}+58451616870107760\,{z}^{62}+
43706535629507712\,{z}^{63}+198265534609662000\,{z}^{64}\\
& +
61858846798214544\,{z}^{65}+566246366845194672\,{z}^{66}+
350533465189882416\,{z}^{67}+1530609927186590640\,{z}^{68}\\
& +
1051600395569647248\,{z}^{69}+5020083336316641840\,{z}^{70}+
1484612323157149056\,{z}^{71}\\
& +13978909783188828972\,{z}^{72}+
8350944317758963440\,{z}^{73}+36744154998139439136\,{z}^{74}\\
& +
25052832953276890320\,{z}^{75}+120253598175729073536\,{z}^{76}+
35073966134587646448\,{z}^{77}\\
& +333202678278582641256\,{z}^{78}+
195412097035559744496\,{z}^{79}+871838586774035783136\,{z}^{80}\\ 
& +
586236291106679233488\,{z}^{81}+2840991256901599362288\,{z}^{82}+
811711787686171246368\,{z}^{83}\\ 
& +7812725956479398246292\,{z}^{84}+
4464414832273941855024\,{z}^{85}+20292794692154281159200\,{z}^{86}\\
& +
13393244496821825565072\,{z}^{87}+65748654802579870955808\,{z}^{88}+
18263515222938853043280\,{z}^{89}\\
& +178982449184800759824144\,{z}^{90}+
98622982203869806433712\,{z}^{91}+460240583618059096690656\,{z}^{92}\\ 
& +
295868946611609419301136\,{z}^{93}+1479344733058047096505680\,{z}^{94}
+394491928815479225734848\,{z}^{95}\\
&+3969575033705759708956908\,{z}^{96
}+2071082626281265935107952\,{z}^{97}+10059544184794720256238624\,{z}^
{98}\\
&+6213247878843797805323856\,{z}^{99}+31953846234053817284522688\,{
z}^{100}+7988461558513454321130672\,{z}^{101}\\
& +
83878846364391270371872056\,{z}^{102}+39942307792567271605653360\,{z}^
{103}\\
& +207700000521349812349397472\,{z}^{104}+
119826923377701814816960080\,{z}^{105}\\
&+647065386239589800011584432\,{z
}^{106}+143792308053242177780352096\,{z}^{107}\\
& +
1635637504105629772251505092\,{z}^{108}+647065386239589800011584432\,{
z}^{109}\\
& +3882392317437538800069506592\,{z}^{110}+
1941196158718769400034753296\,{z}^{111}\\
&+11647176952312616400208519776
\,{z}^{112}+1941196158718769400034753296\,{z}^{113}\\
& +
27176746222062771600486546144\,{z}^{114}+5823588476156308200104259888
\,{z}^{115}\\
& +58235884761563082001042598880\,{z}^{116}+
17470765428468924600312779664\,{z}^{117}\\
&+
157236888856220321402815016976\,{z}^{118}+
314473777712440642805630033952\,{z}^{120}\\
& +
471710666568660964208445050928\,{z}^{122}+
943421333137321928416890101856\,{z}^{124}\\
& +
1415131999705982892625335152784\,{z}^{126}
\end{align*}}

\vspace{1em}

{\it Example $n=9\cdot 20$}: With $x^{9\cdot 14}-1 = (x-1)^9(x+1)^9r_1^9r_2^9r_3^9,r_4^9$, where $r_1,r_2,r_3,r_4$ are prime self-reciprocal polynomials of degrees $2,4,4$ and $8$,
from Theorem \ref{pro:n-even} we obtain 
{\scriptsize \begin{align*}
&\mathcal{G}_{9\cdot 20}^{(3)}(z) = (1+2\sum_{j=1}^53^{j-1}z^{2j-1})^2(1+2\sum_{j=1}^93^{j-1}z^{2j})(1+8\sum_{j=1}^93^{2(j-1)}z^{4j})^2(1+80\sum_{j=1}^93^{4(j-1)}z^{8j})
\end{align*}}
Expanding, for instance from the coefficient of $z^{99}$ we see that the number of $81$-plateaued functions in $\mathcal{D}$ for $p=3$ and $n=9\cdot 20$ is $616946472137940526877139072$. Furthermore we see that the number of bent functions in the set $\mathcal{D}$ is
$\mathcal{N}_{9\cdot 20}^{(3)} = 6054249652811609019026768290053459869736960$.
Here we omit writing down the whole expanded version of the polynomial $\mathcal{G}_{9\cdot 20}^{(3)}$.


\begin{thebibliography}{1}



\bibitem{sb} Berlekamp, E.R., Sloane, N.: The weight enumerator of second-order Reed-Muller codes.
IEEE Trans. Inform. Theory IT-16, 745-751 (1970)

\bibitem{cgl} Carlet, C., Gao, G., Liu, W.: A secondary construction and a transformation 
on rotation symmetric functions, and their action on bent and semi-bent functions. 
J. Combin. Theory, Series A 127, 161--175 (2014) 





\bibitem{cm} \c Ce\c smelio\u glu, A., Meidl, W.: Non-weakly regular bent polynomials from vectorial 
quadratic functions. In: Pott, A., et.al. (eds) Proceedings of the 11th international conference on 
Finite Fields and their Applications, Contemporary Mathematics, 2015, 83--95

%

\bibitem{cpt} Charpin, P., Pasalic, E., Tavernier, C.: On bent and semi-bent quadratic 
Boolean functions. IEEE Trans. Inform. Theory 51, 4286--4298 (2005)

%
%
%


\bibitem{f2} Fitzgerald, R.W.: Trace forms over finite fields of characteristic $2$ with prescribed 
invariants. Finite Fields Appl. 15, 69--81 (2009)


\bibitem{fno} Fu, F.\,W., Niederreiter, H., \"{O}zbudak, F.:
Joint linear complexity of multisequences consisting of linear recurring sequences.
Cryptogr. Commun. 1, 3--29 (2009)



\bibitem{hk} Helleseth, T., Kholosha, A.: Monomial and quadratic bent functions over the finite
fields of odd characteristic. IEEE Trans. Inform. Theory 52, 2018--2032 (2006)


\bibitem{hf} Hu, H., Feng, D.: On Quadratic bent functions in polynomial forms. IEEE Trans. Inform.
Theory 53, 2610- 2615 (2007)



\bibitem{caw} Ka\c{s}\i kc\i, C., Meidl, W., Topuzo\u glu, A: Spectra of quadratic functions: Average behaviour and counting functions.
Cryptogr. Commun.
DOI 10.1007/s12095-015-0142-9

\bibitem{kgs} Khoo, K., Gong, G., Stinson, D.: A new characterization of semi-bent and bent functions
on finite fields. Designs, Codes, Cryptogr. 38, 279--295 (2006)

\bibitem{kkos} Kocak, N., Kocak, O., \"Ozbudak, F., Saygi, Z.: Characterization and enumeration of a class of semi-bent
Boolean functions. Int. J. Inform. Coding Theory 3, 39--57 (2015)




\bibitem{lhz} Li, S., Hu, L., Zeng, X.: Constructions of $p$-ary quadratic bent functions. 
Acta Appl. Math. 100, 227--245 (2008) 


%



%
\bibitem{mt} Meidl, W., Topuzo\u glu, A.: Quadratic functions with prescribed spectra.
Designs, Codes, Cryptogr. 66, 257--273 (2013)

\bibitem{mrt} Meidl, W., Roy, S., Topuzo\u glu, A.: Enumeration of quadratic functions with prescribed 
Walsh spectrum. IEEE Trans. Inform. Theory 60, 6669--6680 (2014)

%







%
%
\bibitem{yg} Yu, N.Y., Gong, G.: Constructions of quadratic bent functions in polynomial forms.
IEEE Trans. Inform. Theory 52, 3291--3299 (2006)

\end{thebibliography}
\end{document}